\documentclass[12pt]{amsart}
\usepackage{amsmath}
\usepackage{amssymb}
\usepackage{amscd}
\usepackage{multicol}
\usepackage{MnSymbol}
\usepackage{tikz}
\usepackage[OT2,T1]{fontenc}
\usepackage{mathrsfs}
\usepackage{comment} 
\usepackage{url}
\usepackage{tikz-cd}
\usepackage{stmaryrd}
\usepackage[all]{xy}
\usepackage{longtable}
\usepackage{multirow,bigdelim}
\DeclareSymbolFont{cyrletters}{OT2}{wncyr}{m}{n}
 \DeclareMathOperator{\Lie}{Lie}
\DeclareMathOperator{\Ker}{Ker} \DeclareMathOperator{\GL}{GL}
\DeclareMathOperator{\Mat}{Mat} 
 
 \DeclareMathOperator{\Gal}{Gal}

\DeclareMathOperator{\trdeg}{tr\text{.}deg} 
  
\DeclareMathOperator{\Li}{Li}    
\DeclareMathOperator{\rank}{rank}

\DeclareMathOperator{\dv}{div}
\DeclareMathOperator{\ord}{ord}
\DeclareMathOperator{\Span}{Span}

\newcommand{\tr}{\mathrm{tr}}

\title[Linear equations on tensor powers of the Carlitz module]{L\MakeLowercase{inear relations among algebraic points on tensor powers of the }C\MakeLowercase{arlitz module}}
\author{Yen-Tsung Chen and Ryotaro Harada}
\address{Department of Mathematics, Pennsylvania State University, University Park, PA 16802, U.S.A.}

\email{ytchen.math@gmail.com}

\address{Tokyo University of Science, 1-3 Kagurazaka, Assistant Professor's office (3), 15th Floor, Building No. 1, Shinjuku City, Tokyo 1628601, Japan.}
\email{r-harada@rs.tus.ac.jp}
\date{December 1, 2025}

\newtheorem{thm}{Theorem}[section]
\newtheorem{lem}[thm]{Lemma}
\newtheorem{cor}[thm]{Corollary}
\newtheorem{prop}[thm]{Proposition}
\theoremstyle{remark}

\subjclass[2020]{11J72, 11J93, 33E50}
\keywords{Carlitz polylogarithms, tensor powers of the Carlitz module, dual $t$-motive}
\theoremstyle{definition}
\newtheorem{defn}[thm]{Definition}
\newtheorem{rem}[thm]{Remark}

\newtheorem{eg}[thm]{Example}

\numberwithin{equation}{section}
\begin{document}
\bibliographystyle{amsalpha+}

%%%%%%%%%%%%%%%%%%%%%%%%%%%%%%%%%%%%%%%%%%%%%%%%%%%%%%%%%%%%%%%%%%%%%%
\begin{abstract}
    In the present paper, we study linear equations on tensor powers of the Carlitz module using the theory of Anderson dual $t$-motives and a detailed analysis of a specific Frobenius difference equation. As an application, we derive some explicit sufficient conditions for the linear independence of Carlitz polylogarithms at algebraic points in both $\infty$-adic and $v$-adic settings. 
\end{abstract}

\maketitle
\tableofcontents
\setcounter{section}{0}

\section{Introduction}

\subsection{Motivation and the main result}
    Let $\mathbb{G}_m$ be the multiplicative group and let $\alpha_1,\dots,\alpha_\ell\in\mathbb{G}_m(\overline{\mathbb{Q}})=\overline{\mathbb{Q}}^\times$. 
    %be pairwise {\color{red}(delete)} distinct algebraic points.
    Then they are called \emph{multiplicatively dependent} if there exist integers $n_1,\dots,n_\ell$, not all zero, such that
    \[
        \alpha_1^{n_1}\cdots\alpha_\ell^{n_\ell}=1.
    \]
    It is natural to ask how to decide if they are multiplicatively dependent. Moreover, let
    \[
        \mathcal{R}:=\{(n_1,\dots,n_\ell)\in\mathbb{Z}^\ell\mid\alpha_1^{n_1}\cdots\alpha_\ell^{n_\ell}=1\}\subset\mathbb{Z}^\ell.
    \]
    Can one find a set of generators of $\mathcal{R}$ as a $\mathbb{Z}$-module and estimate its {size}? The estimations in this direction were first established by Baker using his celebrated theory of linear forms in logarithms. Another approach is due to Masser, using the geometry of numbers. More precisely, let $K$ be a number field containing $\alpha_1,\dots,\alpha_\ell$. We identify $\mathbb{G}_m(K)$ with $\mathbb{P}^1(K)$ minus two points $(1:0)$ and $(0:1)$ by {identifying} $x$ {with} $(1:x)$. Then we can employ the usual logarithmic height function on $\mathbb{P}^1(K)$. For $\mathbf{x}=(x_1,\dots,x_\ell)\in\mathbb{Z}^\ell$, the size of $\mathbf{x}$ is defined by $|\mathbf{x}|:=\max_{i=1}^\ell\{|x_i|\}$. It was shown in \cite[Theorem~{$\mathbb{G}_m$},~p.~253]{Mas88} that there is a set $\mathbf{m}$ of generators for $\mathcal{R}$ so that the size $|\mathbf{m}|$ is bounded by an explicit quantity which involves the number of algebraic points, the heights of these points, and the cardinality of the roots of unity in $K$. The main result in the present paper is to study an analogue of the multiplicative dependence question for the Carlitz module and its higher dimensional generalizations.
    
    Let $\mathbb{F}_q$ be the finite field with $q$ elements, for $q$ a power of a prime number $p$. Let $t, \theta, X$ be independent variables. Let $A:=\mathbb{F}_q[\theta]$ be the polynomial ring and $k:=\mathbb{F}_q(\theta)$ be its field of fractions. Let $|\cdot|_\infty$ be the normalized non-archimedean absolute value on $k$ so that $|f/g|_\infty:=q^{\deg_{\theta}(f)-\deg_{\theta}(g)}$ for 
    %{\color{red}nonzero}
    $f,g\in A$ with $g\neq 0$. We denote by $k_\infty$ the completion of $k$ with respect to $|\cdot|_\infty$. We further set $\mathbb{C}_\infty$ to be the completion of a fixed algebraic closure $\overline{k_\infty}$ {of $k_{\infty}$} and we fix $\overline{k}$ to be the algebraic closure of $k$ inside $\mathbb{C}_\infty$. For a commutative algebra $R$ containing $A$, we set $R[\tau]$ to be the twisted polynomial ring subject to the relation $\tau\alpha=\alpha^q\tau$ for $\alpha\in R$. Note that $R[\tau]$ can be realized as the $\mathbb{F}_q$-linear endomorphism ring on $R$. 
    
    The Carlitz module is a pair $\mathbf{C}:=(\mathbb{G}_a,[\cdot])$, where $([\cdot]:=a\mapsto[a]):\mathbb{F}_q[t]\to k[\tau]$ is an $\mathbb{F}_q$-algebra homomorphism uniquely determined by $[t]:=\theta+\tau$. We can associate an $\mathbb{F}_q$-linear power series $\exp_\mathbf{C}(X)\in k\llbracket X\rrbracket$ satisfying $\exp_\mathbf{C}(X)\equiv X(\mathrm{mod}~X^q)$ and $\exp_{\mathbf{C}}(a(\theta)X)=[a](\exp_{\mathbf{C}}(X))$ for all $a\in\mathbb{F}_q[t]$. In addition, $\exp_\mathbf{C}(X)$ induces an entire, surjective, $\mathbb{F}_q$-linear map on $\mathbb{C}_\infty$. Furthermore, the following short exact sequences of $\mathbb{F}_q[t]$-modules commute for any $a\in\mathbb{F}_q[t]$
    \begin{equation}\label{Eq:Carlitz_Modules_Uniformization}
        \begin{tikzcd}
            0 \arrow[r] & \Lambda_{\mathbf
        {C}} \arrow[d, "a(\theta)"] \arrow[r] & \mathbb{C}_\infty \arrow[d, "a(\theta)"] \arrow[r, "\exp_{\mathbf{C}}(\cdot)"] & \mathbb{C}_\infty \arrow[d, "{[}a{]}"] \arrow[r] & 0 \\
            0 \arrow[r] & \Lambda_{\mathbf
        {C}} \arrow[r] & \mathbb{C}_\infty \arrow[r, "\exp_{\mathbf{C}}(\cdot)"] & \mathbb{C}_\infty \arrow[r] & 0
        \end{tikzcd}
    \end{equation}
    where $\Lambda_\mathbf{C}:=\Ker(\exp_{\mathbf{C}})=A\Tilde{\pi}$ and $\tilde{\pi}:=\theta (-\theta)^{1/(q-1)}\prod_{i=1}^{\infty} (1-\theta^{1-q^i})^{-1}\in \mathbb{C}_{\infty}^{\times}$ {for some fixed $(q-1)$th root of $-\theta$}. Note that \eqref{Eq:Carlitz_Modules_Uniformization} could be {regarded as} an analogue of the analytic uniformization of the complex multiplicative group $\mathbb{G}_m$. Let $n\in\mathbb{Z}_{>0}$. {The $n$-th tensor power of the Carlitz module, which is a higher dimensional generalization of the Carlitz module,} is the pair $\mathbf{C}^{\otimes n}:=(\mathbb{G}_a^n,[\cdot]_n)$, where $([\cdot]_n:=a\mapsto[a]_n):\mathbb{F}_q[t]\to\Mat_{n}(k[\tau])$ is an $\mathbb{F}_q$-algebra homomorphism uniquely determined by 
    \[
        [t]_n:=\begin{pmatrix}
            \theta & 1 & &  \\
            & \theta & \ddots & \\
            & & \ddots & 1 \\
            \tau & & & \theta
            \end{pmatrix}\in\Mat_{n}(k[\tau]).
    \]
    Note that {$\mathbf{C}^{\otimes 1}$} is simply the Carlitz module. 
    
    We denote by $\mathbf{C}^{\otimes n}(\overline{k})=\mathbb{G}_a^n(\overline{k})$ the $\overline{k}$-valued points on the additive group equipped with the $\mathbb{F}_q[t]$-module structure arising from $[\cdot]_n$. Given $P_1,\dots,P_\ell\in\mathbf{C}^{\otimes n}(\overline{k})$, {we say that} they are \emph{linearly dependent over $\mathbb{F}_q[t]$} if there exist $a_1,\dots,a_\ell\in\mathbb{F}_q[t]$, not all zero, such that
    \[
        [a_1]_nP_1+\cdots+[a_\ell]_nP_\ell=0.
    \]
    With the analogy between the complex multiplicative group $\mathbb{G}_m$ and the Carlitz module $\mathbf{C}$, one can ask an analogue of the multiplicative dependence problem, that is, how to determine {whether} these points are linearly dependent over $\mathbb{F}_q[t]$ {or not}. Moreover, let
    \[
        G:=\{(a_1,\dots,a_\ell)\in\mathbb{F}_q[t]^\ell\mid [a_1]_nP_1+\cdots[a_\ell]_nP_\ell=0\}\subset\mathbb{F}_q[t]^\ell.
    \]
    {One can also ask if we can find a set of generators of $G$ over $\mathbb{F}_q[t]$ and estimate its size.} We give {an affirmative answer as follows}. Let $L\subset\overline{k}$ be a finite extension of $k$, {and} for a divisor $D\in\mathrm{div}(L)$ we set
    \[
        \mathcal{L}(D)=\{f\in L^\times\mid \mathrm{div}(f)+D\geq 0\}\cup\{0\}.
    \]
    For $\mathbf{m}:=(m_1,\dots,m_\ell)\in\mathbb{F}_q[t]^\ell$, we define $\deg_t(\mathbf{m}):=\max_{i=1}^\ell\{\deg_t(m_i)\}$.
    
    \begin{thm}\label{Thm:Estimate_of_Relations}
        Let $L\subset\overline{k}$ be a finite extension of $k$ and $n\in\mathbb{Z}_{>0}$. Let $\mathbf{C}^{\otimes n}=(\mathbb{G}_a^n,[\cdot]_n)$ be the $n$-th tensor power of the Carlitz module and $P_1,\dots,P_\ell\in \mathbf{C}^{\otimes n}(L)$ be distinct non-zero {$L$-valued points}. Then there exists an explicitly constructed divisor $D\in\mathrm{div}(L)$ and $\{\mathbf{m}_1,\dots,\mathbf{m}_\nu\}\subset G$ with bounded degree $\deg_t(\mathbf{m}_i)\leq n(\dim_{\mathbb{F}_q}\mathcal{L}(D)+\ell)$ such that $\rank_{\mathbb{F}_q[t]}G=\nu$ and $G=\mathrm{Span}_{\mathbb{F}_q[t]}\{\mathbf{m}_1,\dots,\mathbf{m}_\nu\}$.
    \end{thm}

   %{\color{blue} 
   We mention that Denis considered this type of linear dependence problem for a family of $t$-modules including tensor powers of the Carlitz module. He established an analogue of Masser's result \cite[Appendix~2]{Den92a} using his theory of canonical height introduced in \cite{Den92b}. In the present paper, we adopt a different method and obtain a new upper bound to this problem. Unlike Denis' upper bound which is given by the analytic invariant from the canonical height, the upper bound established in the present paper only includes the algebraic terms such as the dimension of the Riemann-Roch space $\mathcal{L}(D)$ over $\mathbb{F}_q$. Our approach involves the theory of Anderson's dual $t$-motives and the analysis of Frobenius difference equations. The techniques used here are rooted in \cite{Cha16,KL16} for the study of characteristic $p$ multiple zeta values, and in \cite{Che20a,Ho20} to the study of linear relations among algebraic points on Drinfeld modules. In addition, our strategy could be used to show that the rank of the $\mathbb{F}_q[t]$-module $\mathbf{C}^{\otimes n}(L)$ is countably infinite (see Theorem~\ref{Thm:Linearly_Independence}). This gives an alternative proof of a part of \cite[Cor.3.6.1]{Kua23}, where Kuan's method is based on a generalization of Poonen's work \cite{Poo95} on local height function of certain $t$-modules.
   %} {\color{red}This paragraph would cause confusions. It lets the reader feel that Denis' result already covers yours and you prove the result with different proofs.}
    
\subsection{Carlitz polylogarithms at algebraic points}
    %Let $n\in\mathbb{Z}_{>0}$. The $n$-th polylogarithm is defined by the following power series with rational coefficients:
    %\begin{equation}\label{Eq:Classical_PL}
    %    \mathscr{L}_n(z):=\sum_{m>0}\frac{z^{m}}{m^{n}}\in\mathbb{Q}\llbracket z\rrbracket.
    %\end{equation}
    %It converges for any complex number $z_0\in\mathbb{C}$ with $|z_0|<1$. The study of linear independence for polylogarithms at algebraic points has been developed by using Diophantine approximation and tools from transcendence theory. In a recent work of David, Hirata-Kohno, and Kawashima \cite{DHKK20}, they found a sufficient condition for polylogarithms at algebraic points {to be} linearly independent {over number fields} by extending P\`ade approximation. In particular, for a number field $K$, they give an explicit condition for distinct $\alpha_1,\dots,\alpha_\ell\in K^\times$ and $n\in\mathbb{Z}_{>0}$ so that the set of $n\ell+1$ numbers
    %\[
    %    \{1,\mathscr{L}_1(\alpha_1),\dots,\mathscr{L}_n(\alpha_1),\dots,\mathscr{L}_1(\alpha_\ell),\dots,\mathscr{L}_n(\alpha_\ell)\}
    %\]
    %is a linearly independent set over $K$ (see \cite[Thm.~2.1]{DHKK20} for more details). In fact, their result even works for Lerch functions, a generalization of polylogarithms.

By virtue of Anderson-Thakur \cite{AT90}, it is known that the special values of Carlitz polylogarithm (CPL), and $v$-adic CPL ($v$ is a finite place in $k$) are interpreted by some algebraic points of $\mathbf{C}^{\otimes n}(\overline{k})$.
Then, by giving a slight generalization of this interpretation (see Proposition \ref{Prop:generalization_of_AT90}), we can apply our techniques of proving Theorem \ref{Thm:Estimate_of_Relations} to obtain sufficient conditions of linear/algebraic independence among the special values of CPLs, and $v$-adic CPLs, as below.
%By virtue of the work by Anderson-Thakur \cite{AT90}, we can relate algebraic points of $\mathbf{C}^{\otimes n}$ to the special values of Carlitz polylogarithm (CPL), which we recall as below. }
    
    We set $L_0:=1$ and $L_i:=(\theta-\theta^q)\cdots(\theta-\theta^{q^i})\in A$ for each $i\geq 1$. For $n\in\mathbb{Z}_{>0}$, the $n$-th CPL was defined by Anderson and Thakur \cite{AT90} as follows:
    \begin{equation}\label{Eq:CPL}
        \Li_n(z):=\underset{i\geq 0}{\sum}\frac{z^{q^i}}{L_{i}^{n}}\in k\llbracket z\rrbracket.
    \end{equation}
    Note that CPLs can be viewed as an analogue of classical polylogarithms in positive characteristic. The $n$-th CPL converges at $z_0\in\mathbb{C}_\infty$ with $|z_0|_\infty<|\theta|_\infty^{nq/(q-1)}$. It is known due to Anderson and Thakur \cite{AT90} that the $n$-th CPL appears as the last entry of a logarithm of $\mathbf{C}^{\otimes n}$ at a specific point whenever it is defined (see Section~4.1 for related details). 
    
    Let $\mathbb{L}_0:=1$ and $\mathbb{L}_i:=(t-\theta^q)\cdots(t-\theta^{q^i})\in A[t]$ for each $i\geq 1$. For $n\in\mathbb{Z}_{>0}$ and $\mathbf{f}\in\mathbb{C}_\infty[t]$, we define
    \begin{equation}\label{Eq:t_motivic_CPL}
        \mathcal{L}_{\mathbf{f},n}(t):=\sum_{i\geq 0}\frac{\mathbf{f}^{(i)}}{\mathbb{L}_i^n}\in\mathbb{C}_\infty\llbracket t\rrbracket,
    \end{equation}
    where $\mathbf{f}^{(i)}\in\mathbb{C}_\infty[t]$ is the $i$-th Frobenius twisting defined in \eqref{Eq:Frobenius_twisting}. If $\mathbf{f}=f_0+f_1t+\cdots+f_mt^m\in\mathbb{C}_\infty[t]$ with $f_i\in\mathbb{C}_\infty$ for each $0\leq i\leq m$, we set $\|\mathbf{f}\|:=\max_{0\leq i\leq m}\{|f_i|_\infty\}$. By \cite{CY07}, it is known that $\mathcal{L}_{\mathbf{f},n}$ converges at $t=\theta$ whenever $\|\mathbf{f}\|<|\theta|_\infty^{nq/(q-1)}$. In particular, when $\mathbf{f}=z_0\in\mathbb{C}_\infty$ with $\|\mathbf{f}\|=|z_0|<|\theta|_\infty^{nq/(q-1)}$, we have
    \[
        \mathcal{L}_{z_0,n}|_{t=\theta}=\Li_n(z_0).
    \]
    Thus, $\mathcal{L}_{z_0,n}$ can be regarded as a deformation series of $\Li_n(z_0)$. This suggests that $\mathcal{L}_{\mathbf{f},n}(\theta)$ can be constructed as a CPL at $\mathbf{f}\in\mathbb{C}_\infty[t]$ with $\|\mathbf{f}\|<|\theta|_\infty^{nq/(q-1)}$.
        
    Let $L\subset\overline{k}$ be a finite extension of $k$. We set $M_L$ to be the set of places of $L$. For $w\in M_L$ and $\mathbf{f}=f_0+f_1t+\cdots+f_mt^m$, we define $\ord_w(\mathbf{f}):=\min_{i=0}^m\{\ord_w(f_i)\}$. By using the study of linear equations on tensor powers of the Carlitz module, we get a sufficient condition for linear independence among algebraic points of tensor powers of the Carlitz module. This result is stated as Theorem \ref{Thm:Linearly_Independence}. 
    %We give a generalization of Anderson-Thakur \cite{AT90}, 
    We can relate the algebraic points to the special values of CPLs by Proposition \ref{Prop:generalization_of_AT90}.
    %of $\mathbf{C}^{\otimes n}(\overline{k})$ 
    Then, by applying Yu's transcendence result \cite[Thm.~2.3]{Yu91}, we can show that if algebraic points of 
    %$\mathbf{C}^{\otimes n}(L)$ 
    tensor powers of the Carlitz module are $\mathbb{F}_q[t]$-linearly independent, the related special values of CPLs are $k$-linearly independent (see Lemma \ref{Lem:Lower_Bound}).
    %{\color{red}Then,  }  the transcendence result 
    Combining Theorem \ref{Thm:Linearly_Independence}, Lemma \ref{Lem:Lower_Bound} and the transcendence result by Chang \cite[Thm.~5.4.3]{Cha14}, we obtain 
    %we could deduce 
    the following explicit sufficient condition for the special values of CPLs being linearly independent.
    \begin{thm}\label{Thm:Intro_Application}
        Let $L\subset\overline{k}$ be a finite extension of $k$. Let $n\in\mathbb{Z}_{>0}$ and $\mathbf{f}_i=p^{[i]}_0+p^{[i]}_1(t-\theta)+\cdots+p^{[i]}_{n-1}(t-\theta)^{n-1}\in L[t]$ with $p_j^{[i]}\in L$ and $\|\mathbf{f}_i\|<|\theta|_\infty^{nq/(q-1)}$ for $1\leq i\leq \ell$ and $0\leq j\leq n-1$. Suppose that
        \begin{enumerate}
            \item $\mathbf{f}_1,\dots,\mathbf{f}_\ell$ are linearly independent over $\mathbb{F}_q[t]$,
            \item $\mathrm{ord}_w(\mathbf{f}_i)>0$ for all $1\leq i\leq \ell$, if $w\in M_L$ and $w\mid\infty$,
            \item $\mathrm{ord}_w(\mathbf{f}_i)\geq 1-q$ for all $1\leq i\leq\ell$, if $w\in M_L$ and $w\nmid\infty$.
        \end{enumerate}
        Then
        \[
            \dim_{\overline{k}}\Span_{\overline{k}}\{\mathcal{L}_{\mathbf{f}_1,n}(\theta),\dots,\mathcal{L}_{\mathbf{f}_\ell,n}(\theta)\}=\ell.
        \]
        Moreover, if $\|\mathbf{f}_i\|<|\theta|_\infty^{q/(q-1)}$ for each $1\leq i\leq\ell$, then
        \[
            \dim_{\overline{k}}\Span_{\overline{k}}\{1,\mathcal{L}_{\mathbf{f}_1,1}(\theta),\dots,\mathcal{L}_{\mathbf{f}_\ell,1}(\theta),\dots,\mathcal{L}_{\mathbf{f}_1,n}(\theta),\dots,\mathcal{L}_{\mathbf{f}_\ell,n}(\theta)\}=n\ell+1.
        \]
    \end{thm}

    It is worth mentioning that in the special case $\mathbf{f}_i=\alpha_i\in L$, condition (1) in Theorem~\ref{Thm:Intro_Application} is equivalent to saying $\alpha_1,\dots,\alpha_\ell$ are linearly independent over $\mathbb{F}_q$ since $L$ and $\mathbb{F}_q[t]$ are linearly disjoint over $\mathbb{F}_q$. As an application of \cite[Thm.~4.2]{Mis17} (cf. \cite[Thm.~4.5]{CY07} and \cite[Thm.~6.4.2]{Pap08}), Theorem~\ref{Thm:Intro_Application} implies the following algebraic independence result immediately.
    \begin{cor}\label{Cor:Algebraic_Independence}
        Let $L\subset\overline{k}$ be a finite extension of $k$. Let $n_1,\dots,n_d$ be $d$ distinct positive integers, and let $\mathbf{f}_1,\dots,\mathbf{f}_\ell\in L[t]$ with $\|\mathbf{f}_i\|<|\theta|_\infty^{n_jq/(q-1)}$ for each $1\leq i\leq \ell$ and $1\leq j\leq d$. Suppose that $n_j$ is not divisible by $q-1$ for each $1\leq j\leq d$, $n_i/n_j$ is not an integer power of $p$ for each $i\neq j$, and $\mathbf{f}_1,\dots,\mathbf{f}_\ell$ satisfy all the conditions stated in Theorem~\ref{Thm:Intro_Application}. Then
        \[
            \trdeg_{\overline{k}} \overline{k}(\Tilde{\pi},\mathcal{L}_{\mathbf{f}_1,n_1}(\theta),\dots,\mathcal{L}_{\mathbf{f}_\ell,n_1}(\theta),\dots,\mathcal{L}_{\mathbf{f}_1,n_d}(\theta),\dots,\mathcal{L}_{\mathbf{f}_\ell, n_d}(\theta))=d\ell+1.
        \]
    \end{cor}
    
    %It is known by Chang in \cite[Thm.~5.4.3]{Cha14} that
    %\begin{align*}
    %    \dim_k\Span_k&\{1,\Li_1(\alpha_1),\dots,\Li_1(\alpha_\ell),\dots,\Li_n(\alpha_1),\dots,\Li_n(\alpha_\ell)\}\\
    %    &=1+\sum_{i=1}^n\dim_k\Span_k\{\Li_i(\alpha_1),\dots,\Li_i(\alpha_\ell)\}.
    %\end{align*}
    %Then we could conclude from Theorem~\ref{Thm:Intro_Application} that
    %\[
    %    \dim_k\Span_k\{1,\Li_1(\alpha_1),\dots,\Li_1(\alpha_\ell),\dots,\Li_n(\alpha_1),\dots,\Li_n(\alpha_\ell)\}=n\ell+1.
    %\]
%    \begin{thm}\label{Thm:Intro_Application}
%        Let $L\subset\overline{k}$ be a finite extension of $k$ and $n\in\mathbb{Z}_{>0}$. Let $\alpha_i\in L$ with $|\alpha_i|_\infty<|\theta|_\infty^{q/(q-1)}$ for $1\leq i\leq\ell$. Suppose that
%        \begin{enumerate}
%            \item $\alpha_1,\dots,\alpha_\ell$ are linearly independent over $\mathbb{F}_q$,
%            \item $\mathrm{ord}_w(\alpha_i)>0$ for all $1\leq i\leq\ell$, if $w\in M_L$ and $v\mid\infty$,
%            \item $\mathrm{ord}_w(\alpha_i)\geq 1-q$ for all $1\leq i\leq\ell$, if $w\in M_L$ and $v\nmid\infty$.
%        \end{enumerate}
%        Then
%        \[
%            \dim_k\Span_k\{1,\Li_1(\alpha_1),\dots,\Li_1(\alpha_\ell),\dots,\Li_n(\alpha_1),\dots,\Li_n(\alpha_\ell)\}=n\ell+1.
%        \]
%    \end{thm}

    Let $\epsilon_1,\dots,\epsilon_\ell\in\mathbb{F}_q^\times$ be distinct elements. For each $1\leq i\leq\ell$, consider $\gamma_i$ to be a fixed $(q-1)$-th root of $\epsilon_i$. Then Yeo proved in \cite[Lem.~4.23]{Yeo22} that for any fixed $n\in\mathbb{Z}_{>0}$ and $u_1,\dots,u_\ell\in A$, the set of $\ell$ values $\{\Li_n(\epsilon_1u_1),\dots,\Li_n(\epsilon_\ell u_\ell)\}$ is a $k$-linearly independent set. Due to the condition $(2)$ stated in Theorem~\ref{Thm:Intro_Application}, if we set $L=K(\epsilon_1,\dots,\epsilon_\ell)$, then one sees easily that $\mathrm{ord}_w(\epsilon_iu_i)\leq 0$ for any $u_i\in A$ with $1\leq i\leq\ell$, and $w\in M_L$ with $w\mid\infty$. Thus, Theorem~\ref{Thm:Intro_Application} does not include his result. However, the sufficient conditions stated in Theorem~\ref{Thm:Intro_Application} can be applied to any finite extension of $k$, rather than just the constant field extension $K(\epsilon_1,\dots,\epsilon_\ell)$.
    
    In fact, we also have a parallel $v$-adic version of Theorem~\ref{Thm:Intro_Application} in the following sense. Let $v\in M_k$ be a finite place and $\mathfrak{p}_v\in A$ be the monic irreducible polynomial in $A$ corresponding to $v$. We define the normalized $v$-adic absolute value on $k$ by setting $|\mathfrak{p}_v|_v:=(1/q)^{\deg_{\theta}(\mathfrak{p}_v)}$. Consider the completion $k_v$ of $k$ with respect to $|\cdot|_v$. Let $\mathbb{C}_v$ be the $v$-adic completion of an algebraic closure of $k_v$. Throughout this article, we fix an embedding $\iota_v:\overline{k}\hookrightarrow\mathbb{C}_v$. Note that CPLs converge at $z_0\in\mathbb{C}_v$ with $|z_0|_v<1$ if we regard \eqref{Eq:CPL} as $v$-adic analytic functions on $\mathbb{C}_v$ via the embedding $\iota_v$. Inspired by the work of Anderson and Thakur \cite{AT90}, Chang and Mishiba proposed a way to enlarge the defining domain of $v$-adic CPLs in \cite{CM19} so that $v$-adic CPLs are defined at $z_0\in\mathbb{C}_v$ with $|z_0|_v\leq 1$ (see Section 4.2 for more details). In this case, we denote by $\Li_n(z_0)_v$ its value in $\mathbb{C}_v$. Due to the lack of Chang's result \cite[Thm.~5.4.3]{Cha14} in the $v$-adic setting, we only have the following $v$-adic analogue of Theorem~\ref{Thm:Intro_Application}.

    \begin{thm}\label{Thm:v_adic_Thm}
        Let $L\subset\overline{k}$ be a finite extension of $k$. Let $n\in\mathbb{Z}_{>0}$ and $\alpha_i\in L$ with $|\alpha_i|_v\leq 1$ for $1\leq i\leq \ell$. Suppose that
        \begin{enumerate}
            \item $\alpha_1,\dots,\alpha_\ell$ are linearly independent over $\mathbb{F}_q$,
            \item $\mathrm{ord}_w(\alpha_i)>0$ for all $1\leq i\leq \ell$, if $w\in M_L$ and $w\mid\infty$,
            \item $\mathrm{ord}_w(\alpha_i)\geq 1-q$ for all $1\leq i\leq \ell$, if $w\in M_L$ and $w\nmid\infty$.
        \end{enumerate}
        Then
        \[
            \dim_k\Span_k\{\Li_n(\alpha_1)_v,\dots,\Li_n(\alpha_\ell)_v\}=\ell.
        \]
    \end{thm}

    We end up this subsection with the following few remarks. As an analogue of the classical multiple polylogarithms, for $\mathfrak{s}=(s_1,\dots,s_r)\in(\mathbb{Z}_{>0})^r$, Chang introduced the $\mathfrak{s}$-th Carlitz multiple polylogarithm (CMPL) in \cite{Cha14}, which is defined as follows:
    \begin{equation}\label{Eq:CMPL}
        \Li_{\mathfrak{s}}(z_1,\dots,z_r):=\underset{i_1>\cdots>i_r\geq 0}{\sum}\frac{z_1^{q^{i_1}}\cdots z_r^{q^{i_r}}}{L_{i_1}^{s_1}\cdots L_{i_r}^{s_r}}\in k\llbracket z_1,\dots,z_r\rrbracket.
    \end{equation}
    Linear relations among CMPLs at algebraic points have been studied in \cite{Cha14,CPY19,CH21,CCM22b,IKLNDP22}. 
    But none of them gives explicit sufficient conditions on algebraic points so that the special values of CMPLs with fixed $\mathfrak{s}$ at these points are linearly independent over $k$. It is natural to ask whether we have generalizations of Theorem~\ref{Thm:Intro_Application} for CMPLs at algebraic points. We hope to tackle this problem in the near future.

    {At the time of writing this paper}, there {have been} some  developments concerning similar topics about linear relations among CPLs at algebraic points. For example, using the theory of non-commutative factorization, an abundant family of \emph{Eulerian type} relations involving CPLs at {algebraic points} {have been} constructed explicitly in \cite[Thm.~C]{Pel23}. On the other hand, it has been shown in \cite[Thm.~A]{GM22} that several linear relations and linearly independence results about CPLs at integral points in $A$ can be explained by the integral motivic cohomology theory introduced {in} \cite{Gaz24}. It would be interesting to compare these results with our approaches presented in this {paper}.

\subsection{Organization}
    The organization of the present article is given as follows. In Section 2, we first follow closely the exposition in \cite{NP21} to review the theory of $t$-modules and Anderson dual $t$-motives. Then we review the essential properties of tensor powers of the Carlitz module. In Section 3, we first establish the equivalence between the existence of $\mathbb{F}_q[t]$-linear relations among algebraic points on tensor powers of the Carlitz module and the existence of solutions of an explicitly constructed Frobenius difference equation. Then we give a proof of Theorem~\ref{Thm:Estimate_of_Relations}, an analogue of Masser's result for tensor powers of the Carlitz module. In Section 4, we recall some results related to CPLs for our later use. Then as an application of the techniques we develop in Section 3, we prove Theorem \ref{Thm:Intro_Application}, Corollary \ref{Cor:Algebraic_Independence}, and Theorem \ref{Thm:v_adic_Thm}. Then we derive a concrete sufficient condition for CPLs at algebraic points to be linearly independent over $k$ in both $\infty$-adic and $v$-adic settings.

\section{Preliminaries}
\subsection{Notation}
    \begin{itemize}
		\setlength{\leftskip}{0.8cm}
        \setlength{\baselineskip}{18pt}
		\item[$\mathbb{F}_q$] :=  A fixed finite field with $q$ elements, for $q$ a power of a prime number $p$.
		\item[$\infty$] :=  A fixed closed point on the projective line $\mathbb{P}^1(\mathbb{F}_q)$.
		\item[$A$] :=  $\mathbb{F}_q[\theta]$, the regular functions of $\mathbb{P}^1$ away from $\infty$.
		\item[$k$] :=  $\mathbb{F}_q(\theta)$, the function field of $\mathbb{P}^1$.
		\item[$k_\infty$] :=  The completion of $k$ at the place $\infty$.
		\item[$\mathbb{C}_\infty$] :=  The completion of a fixed algebraic closure of $k_\infty$.
		\item[$\overline{k}$] := A fixed algebraic closure of $k$ with a fixed embedding into $\mathbb{C}_\infty$.
        \item[$L$]:= A finite extension of $k$. 
        \item[$M_L$] := The set of all places in $L$. 
        \item[$\mathcal{L}(D)$]:=$\{f\in L^{\times}\mid\dv(f)+D\geq 0 \}\cup\{0\}$, the Riemann-Roch space for a divisor $D$ of $L$. 
	\end{itemize}

\subsection{Anderson $t$-modules and dual $t$-motives}

    In this section, we recall the notion of Anderson $t$-modules \cite{And86} and dual $t$-motives \cite[Definition~4.4.1]{ABP04}. For further information on these objects, one can consult \cite{BP20,HJ20,NP21}.
    For $n\in\mathbb{Z}$ and the field of Laurent series $\mathbb{C}_{\infty}(\!(t)\!)$, we define the $n$-fold Frobenius twisting by
    \begin{equation}\label{Eq:Frobenius_twisting}
        \begin{split}
            \mathbb{C}_{\infty}(\!(t)\!)&\rightarrow\mathbb{C}_{\infty}(\!(t)\!)\\
            f:=\sum_{i}a_it^i&\mapsto \sum_{i}a_i^{q^{n}}t^i=:f^{(n)}.
        \end{split}
    \end{equation}
    We denote by $\overline{k}[t, \sigma]$ the non-commutative $\overline{k}[t]$-algebra generated by $\sigma$ subject to the following relation:
    \[
        \sigma f=f^{(-1)}\sigma, \quad f\in\overline{k}[t].
    \]
    Note that $\overline{k}[t,\sigma]$ contains $\overline{k}[t]$, $\overline{k}[\sigma]$, and its center is $\mathbb{F}_q[t]$. Now we recall the notion of Anderson dual $t$-motives.
    
    \begin{defn}[{\cite[Section 4.4.1]{ABP04}}]
        An \emph{Anderson dual $t$-motive} is a left  $\overline{k}[t, \sigma]$-module $\mathcal{M}$ such that
        \begin{itemize}
            \item[(i)] $\mathcal{M}$ is a free left $\overline{k}[t]$-module of finite rank.
            \item[(ii)] $\mathcal{M}$ is a free left $\overline{k}[\sigma]$-module of finite rank.
            \item[(iii)] $(t-\theta)^n \mathcal{M}\subset \sigma \mathcal{M}$ for any sufficiently large integer $n$.
        \end{itemize}
    \end{defn}
    
    We give the following fundamental example of Anderson dual $t$-motives which plays a crucial role in our study of the tensor powers of the Carlitz module.

    \begin{eg}\label{Ex:Carlitz_Tensor_Powers_1}
        Let $n\in\mathbb{Z}_{>0}$. Then the $n$-th tensor power of the Carlitz motive is defined by $C^{\otimes n}=\overline{k}[t]$ with the $\sigma$-action  given by $\sigma f=f^{(-1)}(t-\theta)^n$ for each $f\in\overline{k}[t]$. Note that $C^{\otimes n}$ is a dual $t$-motive, the set $\{1\}$ forms a $\overline{k}[t]$-basis of $C^{\otimes n}$, and the set $\{(t-\theta)^{n-1},\dots,(t-\theta),1\}$ forms a $\overline{k}[\sigma]$-basis of $C^{\otimes n}$.
    \end{eg}
    
    Let $R$ be an $\mathbb{F}_q$-algebra and $\tau:=(x\mapsto x^{q}):R\to R$ be the Frobenius $q$-th power operator. We set $R[\tau]$ to be the twisted polynomial ring in $\tau$ over $R$ subject to the relation $\tau\alpha=\alpha^q\tau$ for $\alpha\in R$. Now we recall the notion of Anderson $t$-modules.
    \begin{defn}[{\cite[Section 1.1]{And86}}]
        Let $L\subset\overline{k}$ be a field containing $A$ and $d\in\mathbb{Z}_{>0}$. A $d$-dimensional $t$-module defined over $L$ is a pair $E=(\mathbb{G}_a^d,\rho)$ where $\mathbb{G}_a^d$ is the $d$-dimensional additive group scheme over $L$ and $\rho$ is an $\mathbb{F}_q$-algebra homomorphism
        \begin{align*}
            \rho:\mathbb{F}_q[t]&\rightarrow \Mat_d(L[\tau])\\
                    a&\mapsto \rho_{a}
        \end{align*}
        such that $\partial\rho_t-\theta I_d$ is a nilpotent matrix. Here, for $a\in\mathbb{F}_q[t]$ we define $\partial\rho_a:=\alpha_0$ whenever $\rho_a=\alpha_0+\sum_{i\geq 1}\alpha_i\tau^i$ for $\alpha_i\in\Mat_d(L)$.
    \end{defn}
    Let $F$ be a subfield of $\overline{k}$ containing $L$. Then we denote by $E(F)$ the $F$-valued points of the Anderson $t$-module $E$ defined over $L$. More precisely, it is a pair $(\mathbb{G}_a^d(F),\rho)$ of the $F$-valued points of $\mathbb{G}_a^d$ and the homomorphism $\rho$ defined over $L$. Given a $d$-dimensional Anderson $t$-module $E=(\mathbb{G}_a^d, \rho)$ over $L$, Anderson \cite{And86} (see \cite[Lem.~5.9.3]{Gos96} and \cite[Rem.~2.2.3]{NP21} for related discussions) showed that there {exists a unique} $d$-variable $\mathbb{F}_q$-linear power series 
    %{\color{red} Rewrite this part.}
    %\[
    %    \exp_E\in\overline{k_\infty}\llbracket z_1,\dots,z_d\rrbracket^d
    %\]
    %so that
    \[
        \exp_E\begin{pmatrix}
            z_1\\
            \vdots\\
            z_d
        \end{pmatrix}=\begin{pmatrix}
            z_1\\
            \vdots\\
            z_d
        \end{pmatrix}+\sum_{i\geq 1}Q_i\begin{pmatrix}
            z_1^{q^i}\\
            \vdots\\
            z_d^{q^i}
        \end{pmatrix},~Q_i\in\Mat_{d}(L)
    \]
    {such that,} as formal power series, the following identity holds for any $a\in\mathbb{F}_q[t]$
    \[
        \rho_a\circ \exp_E=\exp_E\circ \partial\rho_a.
    \]
    The power series $\exp_E$ is called the \emph{exponential map} of $E$. 
    %We mention that the exponential map $\exp_E$ of $E$ is in fact defined over $L$, namely $\exp_E\in L\llbracket z_1,\dots,z_d\rrbracket^d$. This is due to Anderson \cite[Prop.~2.1.4,~Lem.~2.1.6]{And86} and noted by Goss \cite[Lem.~5.9.3]{Gos96} (see also \cite[Rem.~2.2.3]{NP21}). 
    The kernel of the exponential map $\exp_E$ is called the \emph{period lattice} of $E$ and is denoted by $\Lambda_E$. 
    %If $\exp_E$ is a surjective map on $\mathbb{C}_\infty^d$, then $E$ is called \emph{uniformizable}.
    The formal inverse of $\exp_E$ is denoted by $\log_E$ and is called the \emph{logarithm map} of $E$. As a formal power series, $\log_E$ is of the form
    \[
        \log_E\begin{pmatrix}
            z_1\\
            \vdots\\
            z_d
        \end{pmatrix}=\begin{pmatrix}
            z_1\\
            \vdots\\
            z_d
        \end{pmatrix}+\sum_{i\geq 1}P_i\begin{pmatrix}
            z_1^{q^i}\\
            \vdots\\
            z_d^{q^i}
        \end{pmatrix},~P_i\in\Mat_{d}(L).
    \]
    Moreover, as formal power series, the following identity holds for any $a\in\mathbb{F}_q[t]$
    \[
        \log_E\circ\rho_a=\partial\rho_a\circ\log_E. 
    \]
    
    In what follows, we explain how to associate a $\overline{k}[t,\sigma]$-module $\mathcal{M}_E$ for {a} given $t$-module $E=(\mathbb{G}_a^d,\rho)$ over $L$. We set $\mathcal{M}_E:=\Mat_{1\times d}(\overline{k}[\sigma])$. It naturally has a left $\overline{k}[\sigma]$-module structure. The left $\overline{k}[t]$-module {structure} of $\mathcal{M}_E$ is given {in the following way}: for each $m\in\mathcal{M}_E$, we define
    \begin{equation}
        tm:=m\rho_t^\star.
    \end{equation}
    Here we define
    \begin{equation}
        \rho_t^\star:=\alpha_0^\tr+(\alpha_1^{(-1)})^\tr\sigma+\cdots+(\alpha_s^{(-r)})^\tr\sigma^s
    \end{equation}
    whenever $\rho_t=\alpha_0+\sum_{i=1}^s\alpha_i\tau^i$ with $\alpha_i\in\Mat_d(L)$. It is clear that $\mathcal{M}_E$ is free of rank $d$ over $\overline{k}[\sigma]$ and a straightforward computation shows that
    \[
        (t-\theta)^d\mathcal{M}_E\subset\sigma\mathcal{M}_E.
    \]
    If $\mathcal{M}_E$ is free of finite rank over $\overline{k}[t]$, namely it defines an Anderson dual $t$-motive, then the $t$-module $E$ is called \emph{$t$-finite}. In this case, we define $r:=\rank_{\overline{k}[t]}\mathcal{M}_E$ to be the rank of the dual $t$-motive $\mathcal{M}_E$.
    
    We now explain how to recover a $t$-finite $t$-module $E=(\mathbb{G}_a^d,\rho)$ from its associated Anderson dual $t$-motive $\mathcal{M}_E=\Mat_{1\times d}(\overline{k}[\sigma])$. Let $m=\sum_{i=0}^{n}\mathbf{\alpha}_i\sigma^i\in\mathcal{M}_E$ with $\mathbf{\alpha}_i\in\Mat_{1\times d}(\overline{k})$. Then we define
    \[
        \epsilon_0(m):=\alpha_0^{\tr},~
        \epsilon_1(m):=\left(\sum_{i=0}^n\mathbf{\alpha}_i^{(i)}\right)^\tr\in\overline{k}^d.
    \]
    Note that $\epsilon_0:\mathcal{M}_E\to\overline{k}^d$ is a $\overline{k}$-linear map and $\epsilon_1:\mathcal{M}_E\to\overline{k}^d$ is an $\mathbb{F}_q$-linear map. We have the following lemma due to Anderson.
    
    \begin{lem}[Anderson,~{\cite[Prop.~2.5.8]{HJ20}},~{\cite[Lem.~3.1.2]{NP21}}]\label{Lem:DtMotives_to_tModules}
        For any $a\in\mathbb{F}_q[t]$, we have the following commutative diagrams with exact rows:
        \[
        \begin{tikzcd}
            0 \arrow[r] & \mathcal{M}_E \arrow[d, "a(\cdot)"] \arrow[r, "\sigma(\cdot)"] & \mathcal{M}_E \arrow[d, "a(\cdot)"] \arrow[r, "\epsilon_0"] & \overline{k}^d \arrow[d, "\partial\rho_a(\cdot)"] \arrow[r] & 0 \\
            0 \arrow[r] & \mathcal{M}_E \arrow[r, "\sigma(\cdot)"] & \mathcal{M}_E \arrow[r, "\epsilon_0"] & \overline{k}^d \arrow[r] & 0
        \end{tikzcd}
        \]
        and
        \[
        \begin{tikzcd}
            0 \arrow[r] & \mathcal{M}_E \arrow[d, "a(\cdot)"] \arrow[r, "(\sigma-1)(\cdot)"] & \mathcal{M}_E \arrow[d, "a(\cdot)"] \arrow[r, "\epsilon_1"] & \overline{k}^d \arrow[d, "\rho_a(\cdot)"] \arrow[r] & 0 \\
            0 \arrow[r] & \mathcal{M}_E \arrow[r, "(\sigma-1)(\cdot)"] & \mathcal{M}_E \arrow[r, "\epsilon_1"] & \overline{k}^d \arrow[r] & 0.
        \end{tikzcd}
        \]
        In particular, $\epsilon_0$ and $\epsilon_1$ induce isomorphisms of $\overline{k}[t]$-modules and $\mathbb{F}_q[t]$-modules respectively:
        \[
            \epsilon_0:\mathcal{M}_E/\sigma\mathcal{M}_E\cong \Lie(E)(\overline{k}),~
            \epsilon_1:\mathcal{M}_E/(\sigma-1)\mathcal{M}_E\cong E(\overline{k})
        \]
        where {by} $\Lie(E)(\overline{k})$ we mean the Lie algebra $\Lie\mathbb{G}_a^d(\overline{k})=\overline{k}^d$ equipped with the $\mathbb{F}_q[t]$-module structure coming from $(\partial\rho:=a\mapsto\partial\rho_a):\mathbb{F}_q[t]\to\Mat_{d}(\overline{k})$.
    \end{lem}
    
 %   Now we recall the notion of $t$-frame. Let $\{m_1,\dots,m_r\}$ be a $\overline{k}[t]$-basis of $\mathcal{M}_E$. Then there is an unique matrix $\Phi_E\in\Mat_r(\overline{k}[t])$ such that
  %  \[
   %     %\sigma(m_1,\dots,m_r)^\tr=\Phi_E(m_1,\dots,m_r)^\tr.
    %\]
    %Now we define a map
    %\begin{align*}
     %   \iota:\Mat_{1\times r}(\overline{k}[t])&\to\mathcal{M}_E\\
     %   (a_1,\dots,a_r)&\mapsto a_1m_1+\cdots+a_rm_r.
    %\end{align*}
    %We call the pair $(\iota,\Phi_E)$ a $t$-frame of $E$. Note that for any $(a_1,\dots,a_r)\in\Mat_{1\times r}(\overline{k}[t])$, we have
    %\begin{align*}
        %\sigma\bigl(\iota(a_1,\dots,a_r)\bigr)&=\sigma(a_1,\dots,a_r)(m_1,\dots,m_r)^\tr\\
        %&=(a_1^{(-1)},\dots,a_r^{(-1)})\sigma(m_1,\dots,m_r)^\tr\\
        %&=(a_1^{(-1)},\dots,a_r^{(-1)})\Phi_E(m_1,\dots,m_r)^\tr\\
        %&=\iota\bigl((a_1^{(-1)},\dots,a_r^{(-1)})\Phi_E\bigr).
    %\end{align*}

\subsection{Tensor powers of the Carlitz module}
    In what follows, we recall the definition of the tensor powers of the Carlitz module.
    
    \begin{defn}\label{Def:Carlitz_Tensor_Powers}
        Let $n\in\mathbb{Z}_{>0}$. Then the $n$-th tensor power of the Carlitz module is the $t$-module given by the pair $\mathbf{C}^{\otimes n}:=(\mathbb{G}_a^n,[\cdot]_n)$, where $[\cdot]_n$ is uniquely determined by
        \[
            [t]_n:=
            \begin{pmatrix}
            \theta & 1 & &  \\
            & \theta & \ddots & \\
            & & \ddots & 1 \\
            \tau & & & \theta
            \end{pmatrix}\in\Mat_n(k[\tau]).
        \]
    \end{defn}
    
    Note that the associated $\overline{k}[t,\sigma]$-module is $\mathcal{M}_{\mathbf{C}^{\otimes n}}=\Mat_{1\times n}(\overline{k}[\sigma])$ whose $t$-action on the element $m\in\mathcal{M}_{\mathbf{C}^{\otimes n}}$ is given by
        \[
            tm=m[t]_n^\star=m(
            \begin{pmatrix}
            \theta & & & \\
            1 & \theta & & \\
             & \ddots & \ddots & \\
             & & 1 & \theta
            \end{pmatrix}+
            \begin{pmatrix}
            0 & \cdots & 0 & 1 \\
             & \ddots & & 0 \\
             & & \ddots & \vdots \\
             & & & 0
            \end{pmatrix}\sigma
            ).
        \]
    Let $\{e_i\}_{i=1}^n$ be the standard $\overline{k}[\sigma]$-basis of $\mathcal{M}_{\mathbf{C}^{\otimes n}}$. Recall from Example~\ref{Ex:Carlitz_Tensor_Powers_1} that $C^{\otimes n}=\overline{k}[t]$ is the $\overline{k}[t,\sigma]$-module with $\sigma$-action given by $\sigma f=f^{(-1)}(t-\theta)^n$ for each $f\in\overline{k}[t]$. Since the set $\{(t-\theta)^{n-1},\dots,(t-\theta),1\}$ forms a $\overline{k}[\sigma]$-basis of $C^{\otimes n}$, for $1\leq i\leq n$, the assignment
    \begin{equation}\label{Eq:Identification}
        \begin{split}
            \Theta:\mathcal{M}_{\mathbf{C}^{\otimes n}}&\to C^{\otimes n}\\
        e_i&\mapsto (t-\theta)^{n-i}
        \end{split}
    \end{equation}
    %$\Theta(e_i):=(t-\theta)^{n-i}$ for $1\leq i\leq n$ 
    induces an isomorphism of $\overline{k}[\sigma]$-modules between $\mathcal{M}_{\mathbf{C}^{\otimes n}}$ and the dual $t$-motive $C^{\otimes n}$ we introduced in Example \ref{Ex:Carlitz_Tensor_Powers_1}. In fact, it is straightforward to verify that $\Theta$ is an isomorphism of $\overline{k}[t,\sigma]$-modules, and thus $\mathbf{C}^{\otimes n}$ is $t$-finite.

    One can show that
    %Here we use the fact (See the proof of \cite[Thm. 5.2.1]{CPY19}) that 
    every element $\mathbf{f}\in C^{\otimes n}/(\sigma-1)C^{\otimes n}$ admits the unique expression
    \[
        \mathbf{f}=p_0+p_1(t-\theta)+\cdots+p_{n-1}(t-\theta)^{n-1}+(\sigma-1)C^{\otimes n}\in C^{\otimes n}/(\sigma-1)C^{\otimes n}
    \]
    where $p_i\in\overline{k}$ for $0\leq i\leq n-1$.
    Then according to Lemma~\ref{Lem:DtMotives_to_tModules}, the composition $\epsilon_1\circ\Theta^{-1}$ defines an isomorphism of $\mathbb{F}_q[t]$-modules
    \begin{equation}\label{Eq:Explicit_Correspondence}
        \begin{split}
            (\epsilon_1\circ\Theta^{-1}):C^{\otimes n}/(\sigma-1)C^{\otimes n}&\to\mathbf{C}^{\otimes n}(\overline{k})\\
            p_0+p_1(t-\theta)+\cdots+p_{n-1}(t-\theta)^{n-1}&\mapsto (p_{n-1},\dots,p_0)^\tr.
        \end{split}
    \end{equation}
    An immediate consequence is the following result which gives a collection of generators of $\mathbf{C}^{\otimes n}(L)$. It is also crucial for the later study of the relations between CPLs and the logarithm of $\mathbf{C}^{\otimes n}$.

    \begin{lem}\label{Lem:Generators_of_Carlitz_Tensor_Powers}
        Let $n\in\mathbb{Z}_{>0}$. Let $L\subset\overline{k}$ be a finite extension of $k$. For each $P=(p_{n-1},\dots,p_0)^\tr\in\mathbf{C}^{\otimes n}(L)$, we have
        \[
            P=\sum_{i=0}^{n-1}[t^i]_n(0,\dots,0,f_i)^{\mathrm{tr}}
        \]
        where $f_i=\sum_{j=i}^{n-1}p_j\binom{j}{i}(-\theta)^{j-i}$ for each $0\leq i\leq n-1$.
        In particular, we have
        \[
            \mathbf{C}^{\otimes n}(L)=\mathrm{Span}_{\mathbb{F}_q[t]}\{(0,\dots,0,\alpha)^{\mathrm{tr}}\mid\alpha\in L\}.
        \]
    \end{lem}

    \begin{proof}
        Consider
        \[
            \mathbf{f}_P:=p_0+p_1(t-\theta)+\cdots+p_{n-1}(t-\theta)^{n-1}\in L[t].
        \]
        Note that $\mathbf{f}_P$ can be expressed as
        \[
            \mathbf{f}_P=f_0+f_1t+\cdots+f_{n-1}t^{n-1}\in L[t],
        \]
        where $f_i=\sum_{j=i}^{n-1}p_j\binom{j}{i}(-\theta)^{j-i}\in L$ for each $0\leq i\leq n-1$. On the one hand, we have
        \[
            \left(\epsilon_1\circ\Theta^{-1}\right)(\mathbf{f}_P)=\left(\epsilon_1\circ\Theta^{-1}\right)(p_0+p_1(t-\theta)+\cdots+p_{n-1}(t-\theta)^{n-1})=P.
        \]
        On the other hand, we have
        \begin{align*}
            \left(\epsilon_1\circ\Theta^{-1}\right)(\mathbf{f}_P)&=\left(\epsilon_1\circ\Theta^{-1}\right)(f_0+f_1t+\cdots+f_{n-1}t^{n-1})\\
            &=\sum_{i=0}^{n-1}[t^i]_n\left(\epsilon_1\circ\Theta^{-1}\right)(f_i)\\
            &=\sum_{i=0}^{n-1}[t^i]_n(0,\dots,0,f_i)^{\mathrm{tr}}.
        \end{align*}
        Here the second equality comes from the fact that $\epsilon_1\circ\Theta^{-1}$ is $\mathbb{F}_q[t]$-linear and the last equation follows by \eqref{Eq:Explicit_Correspondence} that $\left(\epsilon_1\circ\Theta^{-1}\right)(f_i)=(0,\dots,0,f_i)^{\mathrm{tr}}$. The desired result now follows.
    \end{proof}

%\begin{rem}
%    With Lemma~\ref{Lem:Generators_of_Carlitz_Tensor_Powers} in hand, to study $\mathbb{F}_q[t]$-linear relations among algebraic points on $\mathbf{C}^{\otimes n}$, it suffices to study points of the form $(0,\dots,0,\alpha)^{\mathrm{tr}}$ for some algebraic element $\alpha$. This elementary fact reduces some of the tasks in the later section.
%\end{rem}

\section{Linear equations on tensor powers of the Carlitz module}
    %The main purpose of this section is to derive some sufficient conditions for Carlitz polylogarithms at algebraic points being linearly independent over $k$.
    
\subsection{The main results}
    In this section, we follow \cite{Che20a} closely to study the $\mathbb{F}_q[t]$-linear relations among algebraic points of tensor powers of the Carlitz module. Throughout this section, we fix $L/k$ to be a finite extension. 
    
    Let $P_1,\dots,P_\ell\in \mathbf{C}^{\otimes n}(L)$. Then we may express 
    \begin{equation}\label{Eq:Points_in_Vector_Form}
        P_i=(p^{[i]}_{n-1},\dots,p^{[i]}_0)^{\rm tr}\in\Mat_{n\times 1}(L).
    \end{equation}
    We set 
    \begin{equation}\label{Eq:Associated_Polynomial_of_Points}
        \mathbf{f}_i:=p^{[i]}_0+p^{[i]}_{1}(t-\theta)+\cdots+p^{[i]}_{n-1}(t-\theta)^{n-1}\in L[t].
    \end{equation}
    Then for each $a_1,\dots,a_\ell\in\mathbb{F}_q[t]$ we have
    \begin{equation}\label{Eq:Linear_Combinations_On_Carlitz_Tensor_Powers}
        \left(\epsilon_1\circ\Theta^{-1}\right)(a_1\mathbf{f}_1+\cdots+a_\ell\mathbf{f}_{\ell})=[a_1]_n(P_1)+\cdots+[a_\ell]_n(P_\ell).
    \end{equation}
    
    We first give an equivalence between the existence of $\mathbb{F}_q[t]$-linear relations among $P_1,\dots,P_\ell$ and the existence of the solution of a specific Frobenius difference equation.
    \begin{lem}\label{Lem:Reduction_Carlitz_Tensor_POwers}
        Given $P_1,\dots,P_\ell\in\mathbf{C}^{\otimes n}(L)$ with the expression in \eqref{Eq:Points_in_Vector_Form}, let $\mathbf{f}_i$ be defined as in \eqref{Eq:Associated_Polynomial_of_Points} for each $1\leq i\leq\ell$. Then the following statements are equivalent.
        \begin{enumerate}
            \item There exist $a_1,\dots,a_\ell\in\mathbb{F}_q[t]$ such that
            \begin{equation}\label{Eq:Linear_Relations_of_CTP}
                [a_1]_n(P_1)+\cdots+[a_\ell]_n(P_\ell)=0.
            \end{equation}
            \item There exist $a_1,\dots,a_\ell\in\mathbb{F}_q[t]$ and $g\in L[t]$ such that
            \begin{equation}\label{Eq:System_of_Difference_Equations_of_CTP}
                g^{(1)}-(t-\theta)^ng=a_1\mathbf{f}_1+\cdots+a_\ell\mathbf{f}_{\ell}.
        \end{equation}
        \end{enumerate}
    \end{lem}
    
    \begin{proof}
        By \eqref{Eq:Linear_Combinations_On_Carlitz_Tensor_Powers} we know that the existence of $a_1,\dots,a_\ell\in\mathbb{F}_q[t]$ such that \eqref{Eq:Linear_Relations_of_CTP} holds is equivalent to $a_1\mathbf{f}_1+\cdots+a_\ell\mathbf{f}_\ell\in (\sigma-1)C^{\otimes n}$.
        The latter statement is equivalent to the existence of $\delta\in\overline{k}[t]$ such that
        \[
            a_1\mathbf{f}_1+\cdots+a_\ell\mathbf{f}_{\ell}=(\sigma-1)\delta=(t-\theta)^n\delta^{(-1)}-\delta.
        \]
        After defining $g:=\delta^{(-1)}$, we see that the existence of 
        $a_1,\dots,a_\ell\in\mathbb{F}_q[t]$ such that \eqref{Eq:Linear_Relations_of_CTP} holds is equivalent to the existence of $a_1,\dots,a_\ell\in\mathbb{F}_q[t]$ and $g\in\overline{k}[t]$ such that $\eqref{Eq:System_of_Difference_Equations_of_CTP}$ holds.
        
        Now it remains to check that if there exist $a_1,\dots,a_\ell\in\mathbb{F}_q[t]$ and $g\in\overline{k}[t]$ such that \eqref{Eq:System_of_Difference_Equations_of_CTP} holds, then $g$ is essentially in $L[t]$. To see this, we apply H.-J. Chen's approach which is also adopted in the proof of \cite[Thm.~2]{KL16}, \cite[Thm.~6.1.1]{Cha16}, and \cite[Lem.~3.1.1]{Che20a}. Let $\mathfrak{F}:=a_1\mathbf{f}_1+\cdots+a_\ell \mathbf{f}_\ell$. By comparing $\deg_t(\cdot)$ on both sides of \eqref{Eq:System_of_Difference_Equations_of_CTP}, if $\deg_t(g)=m$, then $\deg_t(\mathfrak{F})=n+m$. Now we express $g=g_0+g_1t+\cdots+g_mt^m$ and $\mathfrak{F}=\mathfrak{F}_0+\mathfrak{F}_1t+\cdots+\mathfrak{F}_{n+m}t^{n+m}$ where $g_0,\dots,g_m\in\overline{k}$ and $\mathfrak{F}_0,\dots,\mathfrak{F}_{n+m}\in L$. Then \eqref{Eq:System_of_Difference_Equations_of_CTP} implies that
        \begin{align*}
            g_0^q+g_1^qt+\cdots +g_m^qt^m&=\left(g_0+g_1t+\cdots+g_mt^m\right)\sum_{j=0}^n\binom{n}{j}(-\theta)^{n-j}t^j\\
            &+\left(\mathfrak{F}_0+\mathfrak{F}_1t+\cdots+\mathfrak{F}_{n+m}t^{n+m}\right).
        \end{align*}
        By comparing the coefficients of $t^{n+m}$ of both sides, we obtain $0=g_m+\mathfrak{F}_{n+m}$ and thus $g_m\in L$. Then by comparing the coefficients of $t^{n+m-1},t^{n+m-2},\dots,t^n$ inductively, we get the desired result.
    \end{proof}
    
    Based on the previous lemma, we are able to see the connection between the $\mathbb{F}_q[t]$-linear relations among $P_1,\dots,P_\ell$ and the solution space of an explicit $\mathbb{F}_q[t]$-linear system. Recall that for a place $w\in M_L$ and $\mathbf{f}=f_0+f_1t+\cdots+f_mt^m$, we have defined $\ord_w(\mathbf{f})=\min_{i=0}^m\{\ord_w(f_i)\}$.

    \begin{thm}\label{Thm:Linearization_Carlitz}
        Given $P_1,\dots,P_\ell\in\mathbf{C}^{\otimes n}(L)$ with the expression in \eqref{Eq:Points_in_Vector_Form}, let $\mathbf{f}_i$ be defined as in \eqref{Eq:Associated_Polynomial_of_Points} for each $1\leq i\leq\ell$ and
        $$D:=\sum_{w\in M_L}(-C_w)\cdot w\in\mathrm{Div}(L)$$
        where
        $$C_w:=\min_{1\leq i\leq\ell}\{\mathrm{ord}_w(P_i)+(n-1)\mathrm{ord}_w(t-\theta),\lfloor\frac{n}{q-1}\rfloor\mathrm{ord}_w(t-\theta)\}.$$
        Let $d:=\dim_{\mathbb{F}_q}\mathcal{L}(D)$ and $\{\mathfrak{b}_1,\dots,\mathfrak{b}_d\}$ be an $\mathbb{F}_q$-basis of $\mathcal{L}(D)$. Then for $1\leq i\leq \ell$ we express $\mathbf{f}_i=\mathfrak{f}_1^{[i]}\mathfrak{b}_1+\cdots+\mathfrak{f}_d^{[i]}\mathfrak{b}_d$ for some $\mathfrak{f}_1^{[i]},\dots,\mathfrak{f}_d^{[i]}\in\mathbb{F}_q[t]$.
        Moreover, the following statements are equivalent.
        \begin{enumerate}
            \item There exist $a_1,\dots,a_\ell\in\mathbb{F}_q[t]$ such that
            \begin{equation}\label{Eq:Linear_Relations_of_CTP_2}
                [a_1]_n(P_1)+\cdots+[a_\ell]_n(P_\ell)=0.
            \end{equation}
            \item There exist $a_1,\dots,a_\ell\in\mathbb{F}_q[t]$ and $g_1,\dots,g_d\in \mathbb{F}_q[t]$ such that
            \begin{equation}\label{Eq:Linear_System_of_CTP}
                g_1\mathfrak{b}_1^q+\cdots+g_d\mathfrak{b}_d^q-(t-\theta)^n(g_1\mathfrak{b}_1+\cdots +g_d\mathfrak{b}_d)=\left(\sum_{i=1}^\ell a_i\mathfrak{f}_1^{[i]}\right)\mathfrak{b}_1+\cdots+\left(\sum_{i=1}^\ell a_i\mathfrak{f}_d^{[i]}\right)\mathfrak{b}_d.
            \end{equation}
        \end{enumerate}
    \end{thm}
    
    \begin{proof}
        By Lemma~\ref{Lem:Reduction_Carlitz_Tensor_POwers}, it suffices to show that \eqref{Eq:System_of_Difference_Equations_of_CTP} is equivalent to \eqref{Eq:Linear_System_of_CTP}. Let $g\in L[t]$ such that
        \begin{equation}\label{Eq:Difference_Equation_Of_CTP}
            g^{(1)}-(t-\theta)g=a_1\mathbf{f}_1+\cdots+a_\ell\mathbf{f}_\ell.
        \end{equation}
        Let $\mathfrak{F}:=a_1\mathbf{f}_1+\cdots+a_\ell\mathbf{f}_\ell$. We claim that $\mathrm{ord}_w(g)\geq C_w$ and $\mathrm{ord}_w(\mathfrak{F})\geq C_w$ for each $w\in M_L$. We first note that by \eqref{Eq:Points_in_Vector_Form} and \eqref{Eq:Associated_Polynomial_of_Points}, we have
        \begin{align*}
        \mathrm{ord}_w(\mathfrak{F})=\mathrm{ord}_w(a_1\mathbf{f}_1+\cdots+a_\ell\mathbf{f}_\ell)
            &\geq\min_{1\leq i\leq \ell}\{\mathrm{ord}_w(p^{[i]}_0+\cdots+p^{[i]}_{n-1}(t-\theta)^{n-1})\}\\
            &\geq\min_{1\leq i\leq \ell}\{\mathrm{ord}_w(P_i)\}+(n-1)\mathrm{ord}_w(t-\theta)
            \geq C_w.
        \end{align*}
        To see $\mathrm{ord}_w(g)\geq C_w$, we suppose to the contrary that $\mathrm{ord}_w(g)<C_w$. Since $\mathrm{ord}_w(t-\theta)=\min\{\mathrm{ord}_w(1),\mathrm{ord}_w(\theta)\}\leq 0$, we have
        \[
            \mathrm{ord}_w(g)<C_w\leq\mathrm{ord}_w(\mathfrak{F})-n\cdot\mathrm{ord}_w(t-\theta).
        \]
        In particular, we have
        \[
            \mathrm{ord}_w((t-\theta)^ng)=n\cdot\mathrm{ord}_w(t-\theta)+\mathrm{ord}_w(g)<\mathrm{ord}_w(\mathfrak{F}).
        \]
        Then by comparing both sides of \eqref{Eq:Difference_Equation_Of_CTP}, we must have
        \[
            q\cdot\mathrm{ord}_w(g)=n\cdot\mathrm{ord}_w(t-\theta)+\mathrm{ord}_w(g).
        \]
        This implies that
        \[
            n\cdot\mathrm{ord}_w(t-\theta)=(q-1)\cdot\mathrm{ord}_w(g)<(q-1)\cdot C_w\leq (q-1)\cdot\lfloor\frac{n}{q-1}\rfloor\mathrm{ord}_w(t-\theta)
        \]
        which leads to a contradiction. Hence we have $\mathrm{ord}_w(g)\geq C_w$.
        
        Note that we just explained that $g,\mathfrak{F}\in\mathcal{L}(D)\otimes_{\mathbb{F}_q}\mathbb{F}_q[t]$ in the sense that if we regard them as polynomials in variable $t$, then their coefficients are actually in $\mathcal{L}(D)$. Thus, we can express
        \[
            g=g_1\mathfrak{b}_1+\cdots+g_d\mathfrak{b}_d
        \]
        for some $g_1,\dots,g_d\in\mathbb{F}_q[t]$. In particular, \eqref{Eq:Difference_Equation_Of_CTP} becomes
        \[
            g_1\mathfrak{b}_1^q+\cdots+g_d\mathfrak{b}_d^q-(t-\theta)^n(g_1\mathfrak{b}_1+\cdots +g_d\mathfrak{b}_d)=\left(\sum_{i=1}^\ell a_i\mathfrak{f}_1^{[i]}\right)\mathfrak{b}_1+\cdots+\left(\sum_{i=1}^\ell a_i\mathfrak{f}_d^{[i]}\right)\mathfrak{b}_d
        \]
        as desired.
    \end{proof}
    
    To demonstrate Theorem \ref{Thm:Linearization_Carlitz}, we provide the following explicit example.    
    \begin{eg}
We set $n=2$, $q=3$ and $L=k$. Consider the following two points 
\[
P_1=
\left(
    \begin{array}{c}
      0 \\
      \theta^2+1
    \end{array}
\right),
P_2=
\left(
    \begin{array}{c}
      0 \\
      \theta^{-1}
    \end{array}
\right)\in{\bf C}^{\otimes 2}(k).
\]
We claim that $P_1$ and $P_2$ are linearly independent over $\mathbb{F}_3[t]$. To begin with, we choose places $w=\infty, \theta, \theta^2+1$. By the definition of $C_w$, we obtain that 
\[
    C_w=\begin{cases}
        -3& \quad (w=\infty),\\
         -1& \quad (w=\theta),\\
         0& \quad (w=\theta^2+1)
    \end{cases}.
\]
For the divisor $D:=3\cdot(\infty)+1\cdot(\theta)$ of $k$, we set
\[
    \mathcal{L}(D):=\mathcal{L}\left(3\cdot(\infty)+1\cdot(\theta)\right)=\{f\in k^{\times}\mid\ord_\infty(f)\geq -3,~\ord_\theta(f)\geq -1 \}\cup\{0\}.
\]
Then we have
\[
    \mathcal{L}(D)=\mathbb{F}_3\theta^{-1}+\mathbb{F}_3+\mathbb{F}_3\theta+\cdots+\mathbb{F}_3\theta^3.
\]
%the Riemann-Roch space $\mathcal{L}(D)$ with $\mathbb{F}_q$-basis $\mathfrak{b}_1, \mathfrak{b}_2, \ldots, \mathfrak{b}_5$ where $\mathfrak{b}_i=\theta^{i-2}$. 
%Indeed by the definition, each $f\in\mathcal{L}(D)$ satisfies $\dv(f)\geq - D$ which is equivalent to $\ord_w(\dv(f))\geq\ord_w(-D)$. Thus $-1\leq \deg_{\theta}f \leq 3$ and 
%\[
%    \mathfrak{L}(D)=\Span_{\mathbb{F}_q}(\theta^{-1}, 1, \ldots, \theta^3).
%\]
%There exists the divisor $\overline{D}$ of $L$ such that $\mathfrak{b}_i^3, \theta^j\mathfrak{b}_i\in\mathcal{L}(\overline{D})$ for all $1\leq i\leq 5$ and $0\leq j\leq 2$.
%give a little more details
%Let $\mathfrak{b}_i=\theta^{i-2}$ for $1\leq i\leq 5$. 
If there are $a_1, a_2\in\mathbb{F}_3[t]$ such that $[a_1]_2P_1+[a_2]_2P_2=0$, by Theorem \ref{Thm:Linearization_Carlitz}, there exist $g_1,\dots,g_5\in\mathbb{F}_3[t]$ so that the following equation holds:
\[
    g_1(\theta^{-1})^3+g_2+g_3(\theta)^3+\cdots+g_5(\theta^3)^3=(t-\theta)^2(g_1\theta^{-1}+g_2+g_3\theta+\cdots+g_5\theta^3)+a_2\theta^{-1}+a_1+a_1\theta.
\]
By comparing the coefficients of $(\theta^{i})^3$ and $\theta^i$ for $-1\leq i\leq 3$, we obtain 
\[
  \left(
    \begin{array}{ccccccc}
        1 & 0 & 0  & 0 & 0 & 0 & 0 \\
        -t^2 & 0 & 0 & 0 & 0 & 0 & -1  \\
        2t & 1-t^2 & 0 & 0 & 0 & -1 & 0 \\
        -1 & 2t & -t^2 & 0 & 0 & -1 & 0 \\
        0 & -1 & 2t & -t^2 & 0 & 0 & 0 \\
        0 & 0 & 0 & 2t & -t^2 & 0 & 0 \\
        0 & 0 & 0 & -1 & 2t & 0 & 0 \\
        0 & 0 & 0 & 0 & -1 & 0 & 0 \\
        0 & 0 & 0 & 1 & 0 & 0 & 0 \\
        0 & 0 & 0 & 0 & 1 & 0 & 0 
    \end{array}
  \right)
  \left(
  \begin{array}{c} 
  g_1 \\ 
  g_2 \\
  g_3 \\
  g_4 \\
  g_5 \\
  a_1 \\
  a_2
\end{array}
  \right)
  =
   \left(
  \begin{array}{c} 
  0 \\ 
  0 \\
  0 \\
  0 \\
  0 \\
  0 \\
  0
\end{array}
\right).
\]
%resume from here 
%One can check that determinant of the above matrix is non zero, thus all $g_i$ and $a_1, a_2$ must be zero. Therefore $P_1, P_2$ are linearly independent over $\mathbb{F}_q[t]$.   
One can perform the Gauss-Jordan elimination to the above matrix equation. Then we can conclude that $g_1,\dots,g_5$ and $a_1, a_2$ must be zero. Therefore $P_1, P_2$ are linearly independent over $\mathbb{F}_3[t]$.
\end{eg}
    
    Now we are ready to prove Theorem~\ref{Thm:Estimate_of_Relations}.
    %which can be viewed as an analogue of Masser's theorem \cite{Mas88} for tensor powers of the Carlitz module.
    
    \begin{proof}[Proof of Theorem~\ref{Thm:Estimate_of_Relations}]
        Let $D\in\mathrm{div}(L)$ be the divisor of $L$ constructed in Theorem~\ref{Thm:Linearization_Carlitz}. Let $d:=\dim_{\mathbb{F}_q}\mathcal{L}(D)$ and $\{\mathfrak{b}_1,\dots,\mathfrak{b}_d\}$ be an $\mathbb{F}_q$-basis of $\mathcal{L}(D)$. Now we set
        \[
            W:=\mathrm{Span}_{\mathbb{F}_q}\{\mathfrak{b}_i,~\mathfrak{b}_i^q,~\theta^j\mathfrak{b}_i\}_{1\leq i\leq d,~1\leq j\leq n}.
        \]
        Let $\{\lambda_1,\dots,\lambda_{\Tilde{d}}\}\subset W$ be an $\mathbb{F}_q$-basis of $W$, where $\Tilde{d}:=\dim_{\mathbb{F}_q}W$.
        Then we express (\ref{Eq:Linear_System_of_CTP}) as
        \begin{equation}\label{Eq:Linear_System_of_CTP_2}
            Q_1(g_1,\dots,g_d,a_1,\dots,a_\ell)\cdot\lambda_1+\cdots+Q_{\Tilde{d}}(g_1,\dots,g_d,a_1,\dots,a_\ell)\cdot\lambda_{\Tilde{d}}=0
        \end{equation}
        where $Q_i(X_1,\dots,X_{d+\ell})\in\mathbb{F}_q[t][X_1,\dots,X_{d+\ell}]$ is a homogeneous polynomial of degree one in variables $X_1,\dots,X_{d+\ell}$ and $\deg_t(Q_i)\leq 1$ for each $1\leq i\leq \Tilde{d}$. Since $\lambda_1,\dots,\lambda_{\Tilde{d}}$ are $\mathbb{F}_q$-linearly independent, and $\mathbb{F}_q[t]$ and $L$ are linearly disjoint over $\mathbb{F}_q$, we obtain $\lambda_1,\dots,\lambda_{\Tilde{d}}$ are $\mathbb{F}_q[t]$-linearly independent. Thus, $Q_i(g_1,\dots,g_d,a_1,\dots,a_\ell)=0$ for $1\leq i\leq\Tilde{d}$ give rise to an $\mathbb{F}_q[t]$-linear system 
        \[
            \mathfrak{B}\cdot(g_1,\dots,g_d,a_1,\dots,a_\ell)^\tr=0
        \]
        for some $\mathfrak{B}\in\Mat_{m\times s}(\mathbb{F}_q[t])$ with $\deg_t(\mathfrak{B})\leq n$ and $0<m=\rank(\mathfrak{B})<s=d+\ell$. Note that every solution $\mathbf{x}$ of $\mathfrak{B}\mathbf{x}^\tr=0$ gives a solution of (\ref{Eq:Linear_System_of_CTP}) and vice versa. Thus, we have a well-defined surjective $\mathbb{F}_q[t]$-module homomorphism
        \begin{align*}
             \pi:\Gamma:=\{\mathbf{x}\in\mathbb{F}_q[t]^{(d+\ell)}\mid \mathfrak{B}\mathbf{x}^\tr=0\}&\twoheadrightarrow G=\{(a_1,\dots,a_\ell)\in\mathbb{F}_q[t]^\ell\mid \sum_{i=1}^\ell[a_i]_n(P_i)=0\}\\
             (g_1,\dots,g_d,a_1,\dots,a_\ell)&\mapsto (a_1,\dots,a_\ell).
         \end{align*}
         
        By \cite[Cor.~2.2.3]{Che20a}, there exist $\mathbb{F}_q[t]$-linearly independent vectors $\mathbf{x}_1,\dots,\mathbf{x}_{s-m}$ with entries in $\mathbb{F}_q[t]$ such that $\deg_t(\mathbf{x}_i)\leq\rank(\mathfrak{B})\cdot\deg_t(\mathfrak{B})<n(d+\ell)$ and $\mathfrak{B}\cdot\mathbf{x}_i^\tr=0$ for each $1\leq i\leq s-m$. Let $\nu:=\rank_{\mathbb{F}_q[t]}G$. Since $G$ is a free $\mathbb{F}_q[t]$-module of rank $\nu$ and $\pi$ is surjective, there exists an $\mathbb{F}_q[t]$-linearly independent set 
        $$\{\mathbf{m}_1',\dots,\mathbf{m}_\nu'\}\subset\{\pi(\mathbf{x}_1),\dots,\pi(\mathbf{x}_{s-m})\}$$
        such that $G_0:=\mathrm{Span}_{\mathbb{F}_q[t]}\{\mathbf{m}_1',\dots,\mathbf{m}_\nu'\}\subset G$ is of finite index. In other words, we have $\rank_{\mathbb{F}_q[t]}G_0=\rank_{\mathbb{F}_q[t]}G=\nu$. Now we apply \cite[Lem.~2.2.4]{Che20a} and we get an $\mathbb{F}_q[t]$-basis $\{\mathbf{m}_1,\dots,\mathbf{m}_\nu\}$ of $G$ such that $\deg_t(\mathbf{m}_i)\leq\max_{i=1}^\nu\{\deg_t(\mathbf{m}_i')\}<n(d+\ell)$.
    \end{proof}
    
\subsection{A sufficient condition for linear independence}
    Let $L\subset\overline{k}$ be a finite extension of $k$. In this section, we will present a linear independence criterion for a specific family of algebraic points on $\mathbf{C}^{\otimes n}(L)$. As a consequence, we prove that the dimension of the $\mathbb{F}_q(t)$-vector space $\mathbf{C}^{\otimes n}(L)\otimes_{\mathbb{F}_q[t]}\mathbb{F}_q(t)$ is countably infinite. 
    \begin{thm}\label{Thm:Linearly_Independence}
        Let $n\in\mathbb{Z}_{>0}$ and $P_i=(p^{[i]}_{n-1},\dots,p^{[i]}_0)^{\rm tr}\in\mathbf{C}^{\otimes n}(L)$ for $1\leq i\leq\ell$. We set $\mathbf{f}_i:=p^{[i]}_0+\cdots+p^{[i]}_{n-1}(t-\theta)^{n-1}\in L[t]$. Suppose that
        \begin{enumerate}
            \item $\mathbf{f}_1,\dots,\mathbf{f}_\ell$ are linearly independent over $\mathbb{F}_q[t]$,
            \item $\mathrm{ord}_w(\mathbf{f}_i)>0$ for all $1\leq i\leq\ell$, if $w\in M_L$ and $w\mid\infty$,
            \item $\mathrm{ord}_w(\mathbf{f}_i)\geq 1-q$ for all $1\leq i\leq\ell$, if $w\in M_L$ and $w\nmid\infty$.
        \end{enumerate}
        Then
        \[
            \rank_{\mathbb{F}_q[t]}\mathrm{Span}_{\mathbb{F}_q[t]}\{P_1,\dots,P_\ell\}=\ell.
        \]
        In particular, we have
        \[
            \dim_{\mathbb{F}_q(t)}\mathbf{C}^{\otimes n}(L)\otimes_{\mathbb{F}_q[t]}\mathbb{F}_q(t)=\aleph_0.
        \]
    \end{thm}
    
    \begin{proof}
        Suppose on the contrary that there exist $a_1,\dots,a_\ell\in\mathbb{F}_q[t]$ not all zero such that
        \[
            [a_1]_n(P_1)+\cdots+[a_\ell]_n(P_\ell)=0.
        \]
        By Lemma~\ref{Lem:Reduction_Carlitz_Tensor_POwers}, there exists $g\in L[t]$ such that
        \[
            g^{(1)}-(t-\theta)^ng=a_1\mathbf{f}_1+\cdots+a_\ell\mathbf{f}_\ell.
        \]
        Let $F:=a_1\mathbf{f}_1+\cdots+a_\ell\mathbf{f}_\ell\in L[t]$. Since $\mathbf{f}_1,\dots,\mathbf{f}_\ell\in L[t]$ are linearly independent over $\mathbb{F}_q[t]$
        %and , $\alpha_1,\dots,\alpha_\ell\in L[t]$ are linearly independent over $\mathbb{F}_q[t]$ by the linear disjointness over $\mathbb{F}_q$ of $\mathbb{F}_q[t]$ and $L$. Then 
        we have $F\neq 0$ in $L[t]$ and hence $g\neq 0$ in $L[t]$.
        
        Let $w\in M_L$ with $w\mid\infty$. Then
        \[
            \mathrm{ord}_w(F)=\mathrm{ord}_w(a_1\mathbf{f}_1+\cdots +a_\ell\mathbf{f}_\ell)\geq\min\{\mathrm{ord}_w(\mathbf{f}_i)\}_{i=1}^\ell>0.
        \]
        Thus, if we express $F=F_0+F_1t+\cdots+F_mt^m$ where $m\in\mathbb{Z}_{\geq 0}$ and $F_i\in L$ for $0\leq i\leq m$, then $\mathrm{ord}_w(F_i)>0$ for each $0\leq i\leq m$ and $w\in M_L$ with $w\mid\infty$. Since $F\neq 0$ in $L[t]$, we may assume that $F_m\neq 0$. 
        Then there exists $w\in M_L$ with $w\nmid\infty$ such that 
        \[
            \mathrm{ord}_w(F)=\min\{\mathrm{ord}_w(F_i)\}_{i=0}^m\leq\mathrm{ord}_w(F_m)<0.
        \]
        Note that we must have $\mathrm{ord}_w(g)<0$, otherwise
        \begin{align*}
            \mathrm{ord}_w(F)&=\mathrm{ord}_w(g^{(1)}-(t-\theta)^ng)\\
            &\geq\min\{q\cdot\mathrm{ord}_w(g),n\cdot\mathrm{ord}_w(t-\theta)+\mathrm{ord}_w(g)\}\\
            &\geq 0
        \end{align*}
        which leads to a contradiction. 
        
        On the other hand, the fact that $\mathrm{ord}_w(g)<0$ implies that
        \[
            \min\{q\cdot\mathrm{ord}_w(g),n\cdot\mathrm{ord}_w(t-\theta)+\mathrm{ord}_w(g)\}=q\cdot\mathrm{ord}_w(g)
        \]
        and hence $\mathrm{ord}_w(F)=q\cdot\mathrm{ord}_w(g)$. Now the inequality
        \[
            0>q\cdot\mathrm{ord}_w(g)=\mathrm{ord}_w(F)=\mathrm{ord}_w(a_1\mathbf{f}_1+\cdots+a_\ell\mathbf{f}_\ell)\geq 1-q
        \]
        leads to a contradiction {because} $\mathrm{ord}_w(g)\in\mathbb{Z}$.
        
        Finally, we are going to prove that 
        \[
            \dim_{\mathbb{F}_q(t)}\mathbf{C}^{\otimes n}(L)\otimes_{\mathbb{F}_q[t]}\mathbb{F}_q(t)=\aleph_0.
        \]
        Since $L$ is a countable set, it is clear that $\dim_{\mathbb{F}_q(t)}\mathbf{C}^{\otimes n}(L)\otimes_{\mathbb{F}_q[t]}\mathbb{F}_q(t)\leq\aleph_0$. On the other hand, since $\mathbf{C}^{\otimes n}(k)\subset\mathbf{C}^{\otimes n}(L)$, it is enough to show that $\dim_{\mathbb{F}_q(t)}\mathbf{C}^{\otimes n}(k)\otimes_{\mathbb{F}_q[t]}\mathbb{F}_q(t)\geq\aleph_0$.
        
        Let $w\in M_k$ be a finite place and let $f_w\in A$ be the monic irreducible polynomial associated to $w$. We claim that any finite non-empty subset $S$ of $\{(0,\dots,0,f_w^{-1})^\tr\}_{w\in M_k\setminus\{\infty\}}\subset\mathbf{C}^{\otimes n}(k)$ is an $\mathbb{F}_q[t]$-linearly independent set. Indeed, it is clear that $\{f_w^{-1}\}_{w\in M_k\setminus\{\infty\}}$ is an $\mathbb{F}_q$-linearly independent set. Also, $\ord_\infty(f_w^{-1})=\deg_\theta(f_w)>0$ and $\ord_w(f_w^{-1})=-1$ imply that $S$ is an $\mathbb{F}_q[t]$-linearly independent set by the first part of the proof. Hence
        \[
            \aleph_0=\rank_{\mathbb{F}_q[t]}\mathrm{Span}_{\mathbb{F}_q[t]}\{(0,\dots,0,f_v^{-1})^\tr\}_{v\in M_k\setminus\{\infty\}}\leq\dim_{\mathbb{F}_q(t)}\mathbf{C}^{\otimes n}(k)\otimes_{\mathbb{F}_q[t]}\mathbb{F}_q(t).
        \]
        Now the desired result follows immediately.
    \end{proof}

\section{Carlitz polylogarithms}
    The main purpose of this section is to derive some sufficient conditions for Carlitz polylogarithms at algebraic points to be linearly independent over $\overline{k}$.
\subsection{Linear relations among CPLs at algebraic points}
    %Set $L_0:=1$ and $L_i:=(\theta-\theta^q)\cdots(\theta-\theta^{q^i})$ for each $i\geq 1$. 
    For $n\in\mathbb{Z}_{>0}$, we recall that the $n$-th Carlitz polylogarithms (CPLs) is given in \eqref{Eq:CPL}.
    %\begin{equation}\label{Eq:CPL}
    %    \Li_n(z):=\underset{i\geq 0}{\sum}\frac{z^{q^i}}{L_{i}^{n}}\in k\llbracket z\rrbracket.
    %\end{equation}    
    %CPLs can be viewed as an analogue of classical polylogarithms 
    %$$\mathscr{L}_n(z):=\sum_{m>0}\frac{z^{m}}{m^{n}}\in\mathbb{Q}\llbracket z\rrbracket$$
    %in the positive characteristic setting. 
    It is due to Anderson-Thakur \cite{AT90} that if $\alpha\in\mathbb{C}_\infty$ with $|\alpha|_\infty<|\theta|_\infty^{nq/(q-1)}$, then
    \begin{equation}\label{Eq:Log_Interpretation_of_CPL}
        \log_{\mathbf{C}^{\otimes n}}
        \begin{pmatrix}
            0 \\
            \vdots \\
            0 \\
            \alpha
        \end{pmatrix}=
        \begin{pmatrix}
            \star \\
            \vdots \\
            \star \\
            \Li_n(\alpha)
        \end{pmatrix}\in\Lie\mathbf{C}^{\otimes n}(\mathbb{C}_\infty).
    \end{equation}

    In what follows, we give a slight generalization of \eqref{Eq:Log_Interpretation_of_CPL} regarding the last entry of $\log_{\mathbf{C}^{\otimes n}}$ at arbitrary algebraic points $P\in\mathbf{C}^{\otimes n}(\overline{k})$. This is simply an application of Lemma~\ref{Lem:Generators_of_Carlitz_Tensor_Powers} and \eqref{Eq:Log_Interpretation_of_CPL}. It also gives an alternate proof of \cite[Thm.~4.1.4]{CCM22a} (see also \cite[Thm.~3.2.10]{Che20b}) in the case of tensor powers of the Carlitz module.
    \begin{prop}\label{Prop:generalization_of_AT90}
        Let $n\in\mathbb{Z}_{>0}$ and $P=(p_{n-1},\dots,p_0)^\tr\in\mathbf{C}^{\otimes n}(\overline{k})$ with $|p_j|_\infty<q^{-j+\frac{nq}{q-1}}$ for each $0\leq j\leq n-1$. If we set $\mathbf{f}_P:=p_0+p_1(t-\theta)+\cdots+p_{n-1}(t-\theta)^{n-1}\in\overline{k}[t]$, then we have
        \[
            \log_{\mathbf{C}^{\otimes n}}(P)=\begin{pmatrix}
                \star\\
                \vdots\\
                \star\\
                \sum_{j=1}^n\sum_{m=0}^{n-j}(-1)^m\binom{n-j}{m}\theta^{n-m-j}\Li_n(\theta^m p_{n-j})
            \end{pmatrix}=\begin{pmatrix}
                \star\\
                \vdots\\
                \star\\
                \mathcal{L}_{\mathbf{f}_P,n}(\theta)
            \end{pmatrix},
        \]
        where $\mathcal{L}_{\mathbf{f}_P,n}(t)$ is the deformation series defined in \eqref{Eq:t_motivic_CPL} and $\mathcal{L}_{\mathbf{f}_P,n}(\theta):=\mathcal{L}_{\mathbf{f}_P,n}(t)\mid_{t=\theta}$.
    \end{prop}
    
    \begin{proof}\label{Prop:Linear_Relations_Among_CPLs}
        Note that by Lemma~\ref{Lem:Generators_of_Carlitz_Tensor_Powers}, we have
        \begin{equation}\label{Eq:Decomposition_of_bv}
            P=\sum_{i=0}^{n-1}[t^i]_n(0,\dots,0,f_i)^{\mathrm{tr}}
        \end{equation}
        where $f_i=\sum_{j=i}^{n-1}p_j\binom{j}{i}(-\theta)^{j-i}$. The condition $|p_j|_\infty<q^{-j+\frac{nq}{q-1}}$ for each $0\leq j\leq n-1$ implies that $|\theta^m p_j|_{\infty}<q^{\frac{nq}{q-1}}$ for each $0\leq j\leq n-1$ and $0\leq m \leq j$. It ensures that all the points $P$ and $(0,\dots,0,f_i)^{\mathrm{tr}}$ for each $0\leq i\leq n-1$ are inside the convergence domain of $\log_{\mathbf{C}^{\otimes n}}$. Then we may apply $\log_{\mathbf{C}^{\otimes n}}(\cdot)$ on both sides of \eqref{Eq:Decomposition_of_bv} to get
        \begin{equation}\label{Eq:Log_of_P}
            \log_{\mathbf{C}^{\otimes n}}(P)=\sum_{i=0}^{n-1}\partial[t^i]_n\log_{\mathbf{C}^{\otimes n}}(0,\dots,0,f_i)^{\mathrm{tr}}.
        \end{equation}
        In particular, the $n$-th coordinate of $\log_{\mathbf{C}^{\otimes n}}(P)$ equals to
        \begin{align*}
            \sum_{i=0}^{n-1}\theta^i\Li_n(f_i)&=\sum_{i=0}^{n-1}\theta^i\Li_n(\sum_{j=i}^{n-1}p_j\binom{j}{i}(-\theta)^{j-i})\\
            &=\sum_{i=0}^{n-1}\sum_{j=i}^{n-1}(-1)^{j-i}\binom{j}{i}\theta^i\Li_n(\theta^{j-i}p_j)\\
            &=\sum_{i=1}^n\sum_{j=1}^i(-1)^{i-j}\binom{n-j}{n-i}\theta^{n-i}\Li_n(\theta^{i-j}p_{n-j})\\
            &=\sum_{j=1}^n\sum_{i=j}^n(-1)^{i-j}\binom{n-j}{n-i}\theta^{n-i}\Li_n(\theta^{i-j}p_{n-j})\\
            &=\sum_{j=1}^n\sum_{m=0}^{n-j}(-1)^m\binom{n-j}{m}\theta^{n-m-j}\Li_n(\theta^m p_{n-j}).
        \end{align*}
        Here the second equality comes from the $\mathbb{F}_q$-linearity of $\Li_n(\cdot)$. In the third equality, we replace $i$ and $j$ by $n-i$ and $n-j$ respectively. The fourth equality is just the change of the order of the summation. The last equality follows by the change of variables $m=i-j$.

        To complete the proof of this proposition, it remains to explain the second equality in the statement. Recall from \eqref{Eq:t_motivic_CPL} that
        \[
            \mathcal{L}_{\mathbf{f}_P,n}(\theta)=\left(\sum_{j\geq 0}\frac{\mathbf{f}_P^{(j)}}{\mathbb{L}_j^n}\right)|_{t=\theta}.
        \]
        Note that $\mathbf{f}_P=f_0+f_1t+\cdots+f_{n-1}t^{n-1}$. Thus,
        \begin{align*}
            \sum_{j\geq 0}\frac{\mathbf{f}_P^{(j)}}{\mathbb{L}_j^n}&=\sum_{i\geq 0}\frac{(f_0+f_1t+\cdots+f_{n-1}t^{n-1})^{(j)}}{\mathbb{L}_j^n}\\
            &=\sum_{i=0}^{n-1}t^i\left(\sum_{j\geq 0}\frac{f_i^{q^j}}{\mathbb{L}_j^n}\right)
        \end{align*}
        By specializing at $t=\theta$, we derive that $\mathcal{L}_{\mathbf{f}_P,n}(\theta)=\sum_{i=0}^{n-1}\theta^i\Li_n(f_i)$. The desired result now follows immediately by comparing with the $n$-th coordinate of the right-hand-side of \eqref{Eq:Log_of_P}.
    \end{proof}

    \begin{rem}
        Note that \cite[Thm.~4.1.4]{CCM22a} and \cite[Thm.~3.2.10]{Che20b} use a different ordering for the coordinate. More precisely, if we set $\mathbf{x}=(x_1,\dots,x_n)^\tr\in\mathbf{C}^{\otimes n}(\overline{k})$, then the $n$-th coordinate of $\log_{\mathbf{C}^{\otimes n}}(\mathbf{x})$ is given by $\sum_{j=1}^n\sum_{m=0}^{n-j}(-1)^m\binom{n-j}{m}\theta^{n-m-j}\Li_n(\theta^m x_j)$ which exactly matches with the formula established in \cite[Thm.~4.1.4]{CCM22a} and \cite[Thm.~3.2.10]{Che20b} for the special case of $\mathbf{C}^{\otimes n}$.
    \end{rem}

    \begin{eg}
        We mention that Proposition~\ref{Prop:Linear_Relations_Among_CPLs} can be used to produce linear relations among CPLs at algebraic points. This is in the same spirit of \cite[Ex.~3.2.13]{Che20b}. Let $q=3$, $n=2$, and $\mathbf{v}=(0,1)^\tr\in\mathbf{C}^{\otimes 2}(k)$. Then Proposition~\ref{Prop:Linear_Relations_Among_CPLs} shows that
        \[
            \log_{\mathbf{C}^{\otimes 2}}([t^2-1]_2\mathbf{v})=\log_{\mathbf{C}^{\otimes 2}}\begin{pmatrix}
                2\theta\\
                \theta^2
            \end{pmatrix}=\begin{pmatrix}
                \star\\
                2\theta\Li_2(\theta)-\Li_2(\theta^2)
            \end{pmatrix}.
        \]
        On the other hand, by the functional equation of $\log_{\mathbf{C}^{\otimes 2}}$, we have
        \[
            \log_{\mathbf{C}^{\otimes 2}}([t^2-1]_2\mathbf{v})=(\theta^2-1)\log_{\mathbf{C}^{\otimes 2}}\begin{pmatrix}
                0\\
                1
            \end{pmatrix}=\begin{pmatrix}
                \star\\
                (\theta^2-1)\Li_2(1)
            \end{pmatrix}.
        \]
        Combining the equations above, we conclude that
        \[
            2\theta\Li_2(\theta)-\Li_2(\theta^2)=(\theta^2-1)\Li_2(1).
        \]
    \end{eg}
    
    By \cite[Thm.~2.3]{Yu91}, it is known that for each non-zero vector $Y=(y_1,\dots,y_n)^{\mathrm{tr}}\in\Lie\mathbf{C}^{\otimes n}(\mathbb{C}_\infty)$ satisfying $\exp_{\mathbf{C}^{\otimes n}}(Y)\in\mathbf{C}^{\otimes n}(\overline{k})$, the $n$-th coordinate $y_n$ is transcendental over $k$. Using this transcendence result together with Proposition~\ref{Prop:Linear_Relations_Among_CPLs}, we can deduce the following result immediately.
    
    \begin{lem}\label{Lem:Lower_Bound}
        Let $n\in\mathbb{Z}_{>0}$ and $P_i=(p^{[i]}_{n-1},\dots,p^{[i]}_0)^{\rm tr}\in\mathbf{C}^{\otimes n}(\overline{k})$ for $1\leq i\leq\ell$ with $|p_j|_\infty<q^{-j+\frac{nq}{q-1}}$ for each $0\leq j\leq n-1$. We set $\mathbf{f}_i:=p^{[i]}_0+\cdots+p^{[i]}_{n-1}(t-\theta)^{n-1}\in \overline{k}[t]$. Then
        \[
            \dim_k\Span_k\{\mathcal{L}_{\mathbf{f}_1,n}(\theta),\dots,\mathcal{L}_{\mathbf{f}_\ell,n}(\theta)\}\geq\rank_{\mathbb{F}_q[t]}\Span_{\mathbb{F}_q[t]}\{P_1,\dots,P_\ell\}
        \]
    \end{lem}
    
    \begin{proof}
        Suppose that there is a non-trivial $k$-linear relations among $\mathcal{L}_{\mathbf{f}_1,n}(\theta),\dots,\mathcal{L}_{\mathbf{f}_\ell,n}(\theta)$, that is, there exist $c_1,\dots,c_\ell\in k$ not all zero such that
        \[
            c_1\mathcal{L}_{\mathbf{f}_1,n}(\theta)+\cdots+c_\ell\mathcal{L}_{\mathbf{f}_\ell,n}(\theta))=0.
        \]
        After multiplying denominators of $c_1,\dots,c_\ell$ if it is necessary, we may assume that all the coefficients $c_1,\dots,c_\ell\in A$. Now by Proposition~\ref{Prop:Linear_Relations_Among_CPLs} we set
        $$Y_i:=\log_{\mathbf{C}^{\otimes n}}(P_i)=
        \begin{pmatrix}
            \star \\
            \vdots \\
            \star \\
            \mathcal{L}_{\mathbf{f}_i,n}(\theta)
        \end{pmatrix}\in \Lie\mathbf{C}^{\otimes n}(\mathbb{C}_\infty)$$
        for each $1\leq i\leq \ell$. We claim that 
        $$X:=[c_1(t)]_nP_1+\cdots+[c_\ell(t)]_nP_n=0\in\mathbf{C}^{\otimes n}(\overline{k}).$$
        Let 
        $$Y:=\partial[c_1(t)]_nY_1+\cdots+\partial[c_\ell(t)]_nY_\ell\in\Lie\mathbf{C}^{\otimes n}(\mathbb{C}_\infty).$$
        Then
        $$\exp_{\mathbf{C}^{\otimes n}}(Y)=[c_1(t)]_n\exp_{\mathbf{C}^{\otimes n}}(Y_1)+\cdots+[c_\ell(t)]_n\exp_{\mathbf{C}^{\otimes n}}(Y_\ell)=X.$$
        If we write $Y=(y_1,\dots,y_n)^{\mathrm{tr}}$, then it is clear from the definition that
        \[
            y_n=c_1\mathcal{L}_{\mathbf{f}_1,n}(\theta)+\cdots+c_\ell\mathcal{L}_{\mathbf{f}_\ell,n}(\theta)=0.
        \]
        Then Yu's theorem \cite[Thm~2.3]{Yu91} implies that $Y$ is a zero vector and hence $X=0$ as desired. As every $k$-linear relation among $\mathcal{L}_{\mathbf{f}_1,n}(\theta),\dots,\mathcal{L}_{\mathbf{f}_\ell,n}(\theta)$ can be lifted to a $\mathbb{F}_q[t]$-linear relation among $P_1,\dots,P_\ell$, the desired inequality now follows immediately.
    \end{proof}

    \begin{rem}
        We mention that the same spirit of Lemma~\ref{Lem:Lower_Bound} was already known in \cite[Thm.~5.1.1]{Cha16}. In fact, we have the following identity
        \begin{equation}\label{Eq:dimension_and_rank}
            \dim_k\Span_k\{\Tilde{\pi}^n,\mathcal{L}_{\mathbf{f}_1,n}(\theta),\dots,\mathcal{L}_{\mathbf{f}_\ell,n}(\theta)\}=1+\rank_{\mathbb{F}_q[t]}\Span_{\mathbb{F}_q[t]}\{P_1,\dots,P_\ell\}.
        \end{equation}
    \end{rem}
    
    By Theorem~\ref{Thm:Linearly_Independence} and Lemma~\ref{Lem:Lower_Bound}, we can deduce Theorem~\ref{Thm:Intro_Application}.
    
    %\begin{thm}\label{Thm:Linearly_Independence_of_CPLs_at_APs}
    %    Let $L\subset\overline{k}$ be a finite extension of $k$. Let $n\in\mathbb{Z}_{>0}$ and $\alpha_i\in L$ with $|\alpha_i|_\infty<|\theta|_\infty^{nq/q-1}$ for $1\leq i\leq \ell$. Suppose that
    %    \begin{enumerate}
    %        \item $\alpha_1,\dots,\alpha_m$ are linearly independent over $\mathbb{F}_q$,
    %        \item $\mathrm{ord}_w(\alpha_i)>0$ for all $1\leq i\leq \ell$, if $w\in M_L$ and $w\mid\infty$,
    %        \item $\mathrm{ord}_w(\alpha_i)\geq 1-q$ for all $1\leq i\leq\ell$, if $w\in M_L$ and $w\not\mid\infty$.
    %    \end{enumerate}
    %    Then
    %    \[
    %        \dim_k\Span_k\{\Li_n(\alpha_1),\dots,\Li_n(\alpha_\ell)\}=\ell.
    %    \]
    %    Moreover, if $|\alpha_i|_\infty<|\theta|_\infty^{q/q-1}$ for each $1\leq i\leq\ell$, then
    %    \[
    %        \dim_k\Span_k\{1,\Li_1(\alpha_1),\dots,\Li_1(\alpha_\ell),\dots,\Li_n(\alpha_1),\dots,\Li_n(\alpha_\ell)\}=n\ell+1.
    %    \]
    %\end{thm}

    \begin{proof}[Proof of Theorem~\ref{Thm:Intro_Application}]
        To prove the first assertion, note that we have
        \begin{align*}
            \ell&\geq \dim_{\overline{k}}\Span_{\overline{k}}\{\mathcal{L}_{\mathbf{f}_1,n}(\theta),\dots,\mathcal{L}_{\mathbf{f}_\ell,n}(\theta)\}\\
            &= \dim_k\Span_k\{\mathcal{L}_{\mathbf{f}_1,n}(\theta),\dots,\mathcal{L}_{\mathbf{f}_\ell,n}(\theta)\}\\
            &\geq \rank_{\mathbb{F}_q[t]}\Span_{\mathbb{F}_q[t]}\{P_1,\dots,P_\ell\}\\
        &=\ell.
        \end{align*}
        Here the first inequality is clear from the counting argument, the second equality follows by \cite[Thm.~5.4.3]{Cha14}, the third inequality comes from Lemma~\ref{Lem:Lower_Bound}, and the last equality follows by Theorem~\ref{Thm:Linearly_Independence}.

        For the second assertion, we first notice that $\Li_n(\mathbf{f}_i)$ is a $k$-linear combination of $\Li_n(\cdot)$ at some explicitly constructed algebraic points by Proposition~\ref{Prop:Linear_Relations_Among_CPLs}. Then \cite[Thm.~5.4.3]{Cha14} implies that
        \begin{align*}
            \dim_{\overline{k}}\Span_{\overline{k}}&\{1,\mathcal{L}_{\mathbf{f}_1,1}(\theta),\dots,\mathcal{L}_{\mathbf{f}_\ell,1}(\theta),\dots,\mathcal{L}_{\mathbf{f}_1,n}(\theta),\dots,\mathcal{L}_{\mathbf{f}_\ell,n}(\theta)\}\\
            &=1+\sum_{i=1}^n\dim_k\Span_k\{\mathcal{L}_{\mathbf{f}_1,i}(\theta),\dots,\mathcal{L}_{\mathbf{f}_\ell,i}(\theta)\}.
        \end{align*}
        The desired result now follows immediately from the first assertion.
    \end{proof}

    \begin{proof}[Proof of Corollary~\ref{Cor:Algebraic_Independence}]
        Recall that $n_1,\dots,n_d$ are $d$ distinct positive integers so that $n_i/n_j$ is not an integral power of $p$ for each $i\neq j$. By \cite[Thm.~4.2]{Mis17} (see also \cite{CY07}), if $\Tilde{\pi}^{n_i},\mathcal{L}_{\mathbf{f}_1,n_i}(\theta),\dots,\mathcal{L}_{\mathbf{f}_\ell,n_i}(\theta)$ are linearly independent over $k$, then the set of $d\ell+1$ elements $\{\Tilde{\pi},\mathcal{L}_{\mathbf{f}_j,n_i}(\theta)\mid 1\leq i\leq d,~1\leq j\leq \ell\}$ are algebraically independent over $\overline{k}$. Since $\mathbf{f}_1,\dots,\mathbf{f}_\ell\in L[t]$ satisfy the sufficient conditions stated in Theorem~\ref{Thm:Intro_Application}, the desired result follows from Theorem~\ref{Thm:Linearly_Independence} together with \eqref{Eq:dimension_and_rank}.
    \end{proof}

\subsection{$v$-adic CPLs at algebraic points}
    %Let $v\in M_k$ be a finite place. We define the normalized $v$-adic absolute value on $k$ by setting $|v|_v=(1/q)^{\deg_{\theta}(v)}$. Consider the $v$-adic completion $k_v$ of $k$ with respect to $|\cdot|_v$. Let $\mathbb{C}_v$ be the $v$-adic completion of the algebraic closure of $k_v$. Throughout this section, we fix an embedding $\iota_v:\overline{k}\hookrightarrow\mathbb{C}_v$. 
    Recall that CPLs converges at $z_0\in\mathbb{C}_v$ with $|z_0|_v<1$ if we regard \eqref{Eq:CPL} as $v$-adic analytic functions on $\mathbb{C}_v$ via the embedding $\iota_v:\overline{k}\to\mathbb{C}_v$ as we mentioned in the introduction. 
    %In this case, we denote by $\Li_n(z_0)_v$ for its value in $\mathbb{C}_v$. Furthermore, 
    Let $\alpha\in\overline{k}$ with $|\alpha|_v<1$. It is known due to Anderson and Thakur \cite{AT90} that \eqref{Eq:Log_Interpretation_of_CPL} still holds in the $v$-adic setting. More precisely, we have
    \begin{equation}\label{Eq:Log_Interpretation_of_vCPL}
        \log_{\mathbf{C}^{\otimes n}}
        \begin{pmatrix}
            0 \\
            \vdots \\
            0 \\
            \alpha
        \end{pmatrix}=
        \begin{pmatrix}
            \ast \\
            \vdots \\
            \ast \\
            \Li_n(\alpha)_v
        \end{pmatrix}\in\Lie\mathbf{C}^{\otimes n}(\mathbb{C}_v).
    \end{equation}
    
    Inspired by the $v$-twist operation proposed by Anderson and Thakur \cite[pg.~187]{AT90}, Chang and Mishiba introduced a way to enlarge the defining domain of CPLs {to $\{\alpha\in\overline{k} \mid |\alpha|_v\leq 1\}$.} To be precise, for $\alpha\in\overline{k}$ with $|\alpha|_v\leq 1$, the $v$-adic CPL at $\alpha$ is defined by
    $$\Li_n(\alpha)_v:=a^{-1}\times n\mbox{-th coordinate of }
    \log_{\mathbf{C}^{\otimes n}}\left([a(t)]_n
        \begin{pmatrix}
            0 \\
            \vdots \\
            0 \\
            \alpha
        \end{pmatrix}\right)\in\Lie\mathbf{C}^{\otimes n}(\mathbb{C}_v)
    $$
    whenever $a\in A$ so that $[a(t)]_n(0,\dots,0,\alpha)^{\mathrm{tr}}\in\mathbf{C}^{\otimes n}(\overline{k})$ lies in the $v$-adic convergence domain of $\log_{\mathbf{C}^{\otimes n}}$. Note that \cite[Prop.~4.1.1]{CM19} guarantees the existence of such element $a\in A$. In addition, Yu's transcendence theorem \cite[Thm.3.7]{Yu91} for the last coordinate of $\log_{\mathbf{C}^{\otimes n}}$ is still valid in the $v$-adic setting. Now we are able to present a $v$-adic analogue of Lemma~ \ref{Lem:Lower_Bound}.
    
    \begin{lem}\label{Lem:Lower_Bound_v}
        Let $\alpha_1,\dots,\alpha_\ell\in\overline{k}$ such that $|\alpha_i|_v\leq 1$ for each $1\leq i\leq \ell$. Then
        $$\dim_k\Span_k\{\Li_n(\alpha_1)_v,\dots,\Li_n(\alpha_\ell)_v\}\geq\rank_{\mathbb{F}_q[t]}\Span_{\mathbb{F}_q[t]}\{
        \begin{pmatrix}
            0 \\
            \vdots \\
            0 \\
            \alpha_1
        \end{pmatrix},\dots,
        \begin{pmatrix}
            0 \\
            \vdots \\
            0 \\
            \alpha_\ell
        \end{pmatrix}
        \}.$$
    \end{lem}
    
    \begin{proof}
        Suppose that there are non-trivial $k$-linear relations among $\Li_n(\alpha_1)_v,\dots,\Li_n(\alpha_\ell)_v$, that is, there exist $c_1,\dots,c_\ell\in k$ which are not all zero such that
        $$c_1\Li_n(\alpha_1)_v+\cdots+c_\ell\Li_n(\alpha_\ell)_v=0.$$
        After multiplying denominators of $c_1,\dots,c_\ell$ if it is necessary, we may assume that all the coefficients $c_1,\dots,c_\ell\in A$. Let $a_i\in A$ so that $[a_i(t)]_n(0,\dots,0,\alpha_i)^{\mathrm{tr}}\in\mathbf{C}^{\otimes n}(\overline{k})$ lies in the $v$-adic convergence domain of $\log_{\mathbf{C}^{\otimes n}}$. We set 
        $$Y_i:=
        \log_{\mathbf{C}^{\otimes n}}\left([a_i(t)]_n
        \begin{pmatrix}
            0 \\
            \vdots \\
            0 \\
            \alpha_i
        \end{pmatrix}\right)\in\Lie\mathbf{C}^{\otimes n}(\mathbb{C}_v)$$
        for each $1\leq i\leq \ell$. We claim that 
        $$X:=[c_1(t)]_n[a_1(t)]_n\begin{pmatrix}
            0 \\
            \vdots \\
            0 \\
            \alpha_1
        \end{pmatrix}+\cdots+[c_\ell(t)]_n[a_\ell(t)]_n\begin{pmatrix}
            0 \\
            \vdots \\
            0 \\
            \alpha_\ell
        \end{pmatrix}=0\in\mathbf{C}^{\otimes n}(\overline{k}).$$
        Let 
        $$Y:=\partial[c_1(t)]_nY_1+\cdots+\partial[c_\ell(t)]_nY_\ell\in\Lie\mathbf{C}^{\otimes n}(\mathbb{C}_v).$$
        Then
        $$\exp_{\mathbf{C}^{\otimes n}}(Y)=[c_1(t)]_n\exp_{\mathbf{C}^{\otimes n}}(Y_1)+\cdots+[c_\ell(t)]_n\exp_{\mathbf{C}^{\otimes n}}(Y_\ell)=X.$$
        If we write $Y=(y_1,\dots,y_n)^{\mathrm{tr}}$, then it is clear from the definition of $Y$ that
        $$y_n=c_1\Li_n(\alpha_1)_v+\cdots+c_\ell\Li_n(\alpha_\ell)_v=0.$$
        Then the $v$-adic version of Yu's theorem \cite[Thm.~3.7]{Yu91} implies that $Y$ is a zero vector and hence $X=0$ as desired. Consequently, we derive
        \begin{align*}
            \dim_k\Span_k\{\Li_n(\alpha_1)_v,\dots,\Li_n(\alpha_\ell)_v\}&\geq\rank_{\mathbb{F}_q[t]}\Span_{\mathbb{F}_q[t]}\{
        [a_1(t)]_n\begin{pmatrix}
            0 \\
            \vdots \\
            0 \\
            \alpha_1
        \end{pmatrix},\dots,
        [a_\ell(t)]_n\begin{pmatrix}
            0 \\
            \vdots \\
            0 \\
            \alpha_\ell
        \end{pmatrix}\}\\
        &=\rank_{\mathbb{F}_q[t]}\Span_{\mathbb{F}_q[t]}\{
        \begin{pmatrix}
            0 \\
            \vdots \\
            0 \\
            \alpha_1
        \end{pmatrix},\dots,
        \begin{pmatrix}
            0 \\
            \vdots \\
            0 \\
            \alpha_\ell
        \end{pmatrix}\}.
        \end{align*}
    \end{proof}

    Now we are ready to present the proof of Theorem~\ref{Thm:v_adic_Thm}. Most of the arguments are parallel to the proof of Theorem~\ref{Thm:Intro_Application}.
    \begin{proof}[Proof of Theorem~\ref{Thm:v_adic_Thm}]
        Note that we have
        \begin{align*}
            \ell&\geq \dim_k\Span_k\{\Li_n(\alpha_1)_v,\dots,\Li_n(\alpha_\ell)_v\}\\
            &\geq \rank_{\mathbb{F}_q[t]}\Span_{\mathbb{F}_q[t]}\{
        \begin{pmatrix}
            0 \\
            \vdots \\
            0 \\
            \alpha_1
        \end{pmatrix},\dots,
        \begin{pmatrix}
            0 \\
            \vdots \\
            0 \\
            \alpha_\ell
        \end{pmatrix}
        \}\\
        &=\ell.
        \end{align*}
        Here the first inequality is clear from the counting argument, the second inequality comes from Lemma~\ref{Lem:Lower_Bound_v}, and the third equality follows by Theorem~\ref{Thm:Linearly_Independence}.
    \end{proof}
    %Now we set $\mathbb{D}_{n,v}^{\mathrm{def}}:=\{x\in\overline{k}\mid |x|_v\leq 1\}$ to be the defining domain of Chang and Mishiba's $v$-adic CPLs at algebraic points. Then 
    %Due to the lack of Chang's result \cite[Thm.~5.4.3]{Cha14} in the $v$-adic setting, we only have the following $v$-adic analogue of Theorem~\ref{Thm:Linearly_Independence_of_CPLs_at_APs}.
    
    %\begin{thm}\label{Thm:v_adic_Thm}
    %    Let $L\subset\overline{k}$ be a finite extension of $k$. Let $n\in\mathbb{Z}_{>0}$ and $\alpha_i\in L$ with $|\alpha_i|_v\leq 1$ for $1\leq i\leq \ell$. Suppose that
    %    \begin{enumerate}
    %        \item $\alpha_1,\dots,\alpha_m$ are linearly independent over $\mathbb{F}_q$,
    %        \item $\mathrm{ord}_w(\alpha_i)>0$ for all $1\leq i\leq m$, if $w\in M_L$ and $w\mid\infty$,
    %        \item $\mathrm{ord}_w(\alpha_i)\geq 1-q$ for all $1\leq i\leq m$, if $w\in M_L$ and $w\not\mid\infty$.
    %    \end{enumerate}
    %    Then
    %    \[
    %        \dim_k\Span_k\{\Li_n(\alpha_1)_v,\dots,\Li_n(\alpha_m)_v\}=m.
    %    \]
    %\end{thm}

\section*{Acknowledgements}
%The authors acknowledge the anonymous referees for making suggestions and comments that have improved the former version of this paper. 
The first author was partially supported by the AMS-Simons Travel Grants. The second author thanks the JSPS Research Fellowships, JSPS Overseas Research Fellowships and the National Center for Theoretical Sciences in Hsinchu for their support during the research project. This work is also supported by JSPS KAKENHI Grant Number JP22J00006.

\end{document}